\documentclass[reqno, eucal]{amsart}
\usepackage{epic,eepic}
\usepackage[mathscr]{eucal} 
\usepackage{psfrag}     

\usepackage{amsmath}
\usepackage{amsthm}

\usepackage[dvips]{graphicx,color}  

\usepackage{amssymb}
\usepackage{epic,eepic}
\usepackage{amscd}
\usepackage{longtable}
\usepackage{array}

\renewcommand{\labelenumi}{(\roman{enumi})}

\newtheorem{theorem}{Theorem}

\newtheorem{corollary}[theorem]{Corollary}

\theoremstyle{definition}

\newtheorem{example}[theorem]{Example}
\newtheorem{remark}[theorem]{Remark}

\newcommand{\Z}{{\mathbb Z}}

\newcommand{\C}{{\mathbb C}}

\newcommand{\Rm}{{\mathscr R}}
\newcommand{\W}{\mathcal{W}}

\begin{document}
\title{Combinatorial Yang-Baxter maps arising from tetrahedron equation}

\author{Atsuo Kuniba}

\address{Institute of Physics, University of Tokyo, Komaba, 153-8902, Japan}

\maketitle

\begin{abstract}
We survey the matrix product solutions of  
the Yang-Baxter equation obtained recently from the tetrahedron equation.
They form a family of quantum $R$ matrices of
generalized quantum groups interpolating 
the symmetric tensor representations of $U_q(A^{(1)}_{n-1})$ and the
anti-symmetric tensor representations of $U_{-q^{-1}}(A^{(1)}_{n-1})$.
We show that at $q=0$
they all reduce to the Yang-Baxter maps called combinatorial $R$,
and describe the latter by explicit algorithm.
\end{abstract}

\section{Introduction}\label{sec:intro}

Tetrahedron equation \cite{Zam80} is a generalization of the 
Yang-Baxter equation \cite{Bax} and serves as a key to the 
integrability in three dimension (3D).
Typically it has the form called  $RRRR$ type or $RLLL$ type:
\begin{align*}
&\Rm_{1,2,4}\Rm_{1,3,5}\Rm_{2,3,6}\Rm_{4,5,6}=
\Rm_{4,5,6}\Rm_{2,3,6}\Rm_{1,3,5}\Rm_{1,2,4},\\
&{\mathscr L}_{1,2,4}{\mathscr L}_{1,3,5}{\mathscr L}_{2,3,6}\Rm_{4,5,6}=
\Rm_{4,5,6}{\mathscr L}_{2,3,6}{\mathscr L}_{1,3,5}{\mathscr L}_{1,2,4}.
\end{align*}
Here $\Rm \in \mathrm{End}(F^{\otimes 3})$ and 
${\mathscr L} \in \mathrm{End}(V \otimes V \otimes F)$ for some vector spaces
$F$ and $V$.
The above equalities hold in $\mathrm{End}(F^{\otimes 6})$ and 
$\mathrm{End}(V^{\otimes 3}\otimes F^{\otimes 3})$ respectively, and 
the indices specify the components on which 
$\Rm$ and ${\mathscr L}$ act nontrivially.
We call the solutions $\Rm$ and ${\mathscr L}$ 3D $R$ and 3D $L$, respectively.

The tetrahedron equations are reducible to the Yang-Baxter equation 
\begin{align*}
S_{1,2}S_{1,3}S_{2,3} = S_{2,3}S_{1,3}S_{1,2}
\end{align*}
if the spaces $4,5,6$ are evaluated away appropriately \cite{BS}.
Such reductions and the relevant quantum group aspects have been
studied systematically in the recent work \cite{KS,KO2,KOS} 
by Okado, Sergeev and the author 
for the distinguished example of 3D $R$ and 3D $L$ originating in the 
quantized algebra of functions \cite{KV}.
They correspond to the choice 
$V=\C^2$ and the $q$-oscillator Fock space 
$F=\bigoplus_{m \ge 0}\C |m\rangle$.  

In this paper we first review 
the latest development in \cite{KOS} concerning the reduction by trace.
It generates $2^n$ solutions $S(z)$ to the Yang-Baxter equation
from the $n$ product of $\Rm$ and ${\mathscr L}$.
Any $S(z)$ is rational in the parameter $q$ and the (multiplicative) 
spectral parameter $z$.
Symmetry of $S(z)$ is described by
{\em generalized quantum groups} \cite{H,HY}
which include quantum affine \cite{D1,Ji} and super algebras of type $A$.

In the last Section \ref{sec:cr}, 
we supplement a new result, Theorem \ref{th:main}.
It shows that at $q=0$ each $S(z)$ yields a {\em combinatorial $R$}, 
a certain bijection between finite sets satisfying the Yang-Baxter equation.
We describe it by explicit combinatorial algorithm generalizing \cite{NY,HI}.  

The notion of combinatorial $R$ originates 
in the crystal base theory, a theory of quantum groups at $q=0$ \cite{Ka}. 
The motivation for $q=0$ further goes back to
Baxter's corner transfer matrix method \cite{Bax,DJKMO},
where it corresponds to the low temperature limit manifesting 
fascinating combinatorial features of Yang-Baxter integrable lattice models. 
It has numerous applications including 
generalized Kostka-Foulkes polynomials, 
Fermionic formulas of affine Lie algebra 
characters, integrable cellular automata in one dimension and so forth.
See for example \cite{KMN,NY,HKOTY,O,IKT,K} and reference therein.
Combinatorial $R$'s form most systematic examples of 
set-theoretical solutions to the Yang-Baxter equation 
(Yang-Baxter maps) \cite{D2,V} arising from the representation theory of 
quantum groups.

In this paper the 3D $R$ and the 3D $L$ will mainly serve as   
the constituent of the $S(z)$ which  
tends to the combinatorial $R$ at $q=0$.
However they possess a decent combinatorial aspect by themselves 
as pointed out in \cite[eq.(2.41)]{KO1} for the 3D $R$.
In fact their limits (\ref{mrei}) define the maps 
\begin{align*}
\lim_{q\rightarrow 0}\Rm : 
\begin{pmatrix}i \\ j \\ k
\end{pmatrix}\mapsto 
\begin{pmatrix}j+\max(i-k,0)\\
\min(i,k)\\
 j+\max(k-i,0)
\end{pmatrix},
\quad
\lim_{q\rightarrow 0}{\mathscr L}: 
\begin{pmatrix}
i \\ j \\ k
\end{pmatrix}\mapsto 
\begin{pmatrix}
j+\max(i-j-k,0)\\
\min(i,k+j)\\
\max(k+j-i,0)
\end{pmatrix}
\end{align*}
on $(\Z_{\ge 0})^3$ and on
$\{0,1\}\times \{0,1\} \times \Z_{\ge 0}$, respectively.
The tetrahedron equations survive the limit 
nontrivially as the {\em combinatorial tetrahedron equations}, e.g.,

\begin{picture}(270,160)(-29,-80)

\put(231,58){$\Rm_{1,2,4}$}
\put(222,56){\vector(3,-1){24}}

\put(190,56){432361}

\put(159,65){$\Rm_{1,3,5}$}
\put(157,59){\vector(1,0){28}}
\put(122,56){234341}

\put(91,65){$\Rm_{2,3,6}$}
\put(89,59){\vector(1,0){28}}
\put(55,56){261344}

\put(18,60){$\Rm_{4,5,6}$}
\put(27,51){\vector(3,1){24}}

\put(0,40){261435}
\put(27,36){\vector(3,-1){24}}
\put(18,23){$\Rm_{1,2,4}$}
\put(55,20){621835}
\put(89,23){\vector(1,0){28}}
\put(92,11){$\Rm_{1,3,5}$}
\put(122,20){423815}
\put(157,23){\vector(1,0){28}}
\put(160,11){$\Rm_{2,3,6}$}
\put(190,20){432816}
\put(222,29){\vector(3,1){24}}
\put(231,21){$\Rm_{4,5,6}$}
\put(250,40){432361}

\put(0,-80){
\put(231,58){${\mathscr L}_{1,2,4}$}
\put(222,56){\vector(3,-1){24}}

\put(190,56){110354}

\put(159,65){${\mathscr L}_{1,3,5}$}
\put(157,59){\vector(1,0){28}}
\put(122,56){011344}

\put(91,65){${\mathscr L}_{2,3,6}$}
\put(89,59){\vector(1,0){28}}
\put(55,56){011344}

\put(18,60){$\Rm_{4,5,6}$}
\put(27,51){\vector(3,1){24}}

\put(0,40){011435}
\put(27,36){\vector(3,-1){24}}
\put(18,23){${\mathscr L}_{1,2,4}$}
\put(55,20){101535}
\put(89,23){\vector(1,0){28}}
\put(92,11){${\mathscr L}_{1,3,5}$}
\put(122,20){101535}
\put(157,23){\vector(1,0){28}}
\put(160,11){${\mathscr L}_{2,3,6}$}
\put(190,20){110536}
\put(222,29){\vector(3,1){24}}
\put(231,21){$\Rm_{4,5,6}$}
\put(250,40){110354}
}

\end{picture}

They constitute the local relations responsible 
for the Yang-Baxter equation of the combinatorial $R$ 
in Corollary \ref{co:sae}.

The layout of the paper is as follows.
In Section \ref{sec:RL}  we recall the definition of the 3D $R$ and 3D $L$.
In Section \ref{sec:tet}  tetrahedron equations of type $RRRR$ and $RLLL$ are
given with their generalization to $n$-layer case.
In Section \ref{sec:ybe} the $2^n$ family of solutions $S(z)$ to the Yang-Baxter 
equation are constructed by applying the trace reduction.
In Section \ref{sec:qg}  generalized quantum group symmetry 
of $S(z)$ is explained.
Section \ref{sec:cr}  contains the main Theorem \ref{th:main}, which 
describes the combinatorial $R$ arising from $S(z)$ at $q=0$
in terms of explicit combinatorial algorithm.

Throughout the paper we assume that $q$ is not a root of unity and 
use the notations:
\begin{align*}
(z;q)_m = \prod_{k=1}^m(1-z q^{k-1}),\;\;
(q)_m = (q; q)_m,\;\;
\binom{m}{k}_{\!\!q}= \frac{(q)_m}{(q)_k(q)_{m-k}}.
\end{align*}

\section{3D $R$ and 3D $L$}\label{sec:RL}
Let $F$ be a Fock space $F = \bigoplus_{m\ge 0}\C |m \rangle$
and ${\bf a}^\pm, {\bf k} \in \mathrm{End}(F)$ be the operators on it 
called $q$-oscillators:
\begin{align}
{\bf a}^+|m\rangle = |m+1\rangle,\quad
{\bf a}^-|m\rangle = (1-q^{2m})|m-1\rangle,\quad
{\bf k}|m\rangle = q^m|m\rangle.
\label{ac1}
\end{align}
They satisfy the relations
\begin{align}
{\bf k}\,{\bf a}^{\pm} = q^{\pm 1}{\bf a}^{\pm}\,{\bf k},\quad
{\bf a}^+{\bf a}^- = 1-{\bf k}^2,\quad
{\bf a}^-{\bf a}^+ = 1-q^2 {\bf k}^2.
\label{ac2}
\end{align}

We define a three dimensional $R$ operator, 3D $R$ for short,  
$\Rm \in \mathrm{End}(F^{\otimes 3})$ by 
\begin{align}
\Rm(|i\rangle \otimes |j\rangle \otimes |k\rangle) = 
\sum_{a,b,c\ge 0} \Rm^{a,b,c}_{i,j,k}
|a\rangle \otimes |b\rangle \otimes |c\rangle,\label{Rabc}
\end{align}
where several formulas are known for the matrix element:
\begin{align}
&\Rm^{a,b,c}_{i,j,k} =\delta^{a+b}_{i+j}\delta^{b+c}_{j+k}
\sum_{\lambda+\mu=b}(-1)^\lambda
q^{i(c-j)+(k+1)\lambda+\mu(\mu-k)}
\frac{(q^2)_{c+\mu}}{(q^2)_c}
\binom{i}{\mu}_{\!\!q^2}
\binom{j}{\lambda}_{\!\!q^2},\label{Rex1}\\
&\phantom{\Rm^{a,b,c}_{i,j,k}} = \delta^{a+b}_{i+j}\delta^{b+c}_{j+k}
\sum_{\lambda+\mu=b}(-1)^\lambda
q^{ik+b+\lambda(c-a)+\mu(\mu-i-k-1)}\binom{i}{\mu}_{\!q^2}
\binom{\lambda+a}{a}_{\!q^2},\label{Rex2}\\
&\phantom{\Rm^{a,b,c}_{i,j,k}} = \delta^{a+b}_{i+j}\delta^{b+c}_{j+k}\,q^{ik+b}
\oint\frac{du}{2\pi {\mathrm i}u^{b+1}}
\frac{(-q^{2+a+c}u;q^2)_\infty(-q^{-i-k}u;q^2)_\infty}
{(-q^{a-c}u;q^2)_\infty(-q^{c-a}u;q^2)_\infty}. \label{Rex3}
\end{align}
where $\delta^j_{k}=\delta_{j,k}$ just to save the space.
The sum (\ref{Rex1}) is over $\lambda, \mu \in \Z_{\ge 0}$ 
satisfying $\lambda+\mu=b$, $\mu\le i$ and $\lambda \le j$.
The sum (\ref{Rex2}) is over $\lambda, \mu \in \Z_{\ge 0}$ 
satisfying $\lambda+\mu=b$ and $\mu\le i$.
The integral (\ref{Rex3}) encircles $u=0$ anti-clockwise
so as to pick the coefficient of $u^b$.
Derivation of  these formulas can be found in 
\cite[Th.2]{KO1} for (\ref{Rex1}), 
\cite[Sec.4]{KM} for (\ref{Rex2}) and \cite{S} for (\ref{Rex3}).
The 3D $R$ can also be expressed as a collection of operators on the 
third component. For example (\ref{Rex1})  yields
\begin{align}
&\Rm(|i\rangle \otimes |j\rangle \otimes |k\rangle) = 
\sum_{a,b\ge 0} |a\rangle \otimes |b \rangle \otimes 
\Rm^{a,b}_{i,j}|k\rangle,\quad
\Rm^{a,b}_{i,j} \in \mathrm{End}(F),\label{rop}\\
&\Rm^{a,b}_{i,j} = 
\delta^{a+b}_{i+j}\sum_{\lambda+\mu=b}(-1)^\lambda
q^{\lambda+\mu^2-ib}
\binom{i}{\mu}_{\!\!q^2}
\binom{j}{\lambda}_{\!\!q^2}
({\bf a}^-)^\mu({\bf a}^+)^{j-\lambda}
{\bf k}^{i+\lambda-\mu},\label{ropp}
\end{align}
where the sum is taken under the same condition as in (\ref{Rex1}),
which guarantees that the powers of $q$-oscillators are nonnegative.

The 3D $R$ was first obtained as the intertwiner of the quantized coordinate ring 
$A_q(sl_3)$ \cite{KV}\footnote{
The formula for it on p194 in \cite{KV} contains a misprint unfortunately.
The formula (\ref{Rex1}) here is a correction of it.}.  
It was found later also 
from a quantum geometry consideration in a different gauge \cite{BS}. 
They were shown to be the same object in \cite[eq.(2.29)]{KO1}.
See also \cite[App.~A]{KO2} and \cite[Sec.~4]{KM} for 
the recursion relations characterizing $\Rm$ and useful corollaries.
Here we note the properties \cite{KO1}
\begin{align}
&\Rm = \Rm^{-1},\quad 
\Rm^{a,b,c}_{i,j,k}= \Rm^{c,b,a}_{k,j,i} \in q^\xi \Z[q^2],\quad
\Rm^{a,b,c}_{i,j,k}=
\frac{(q^2)_i(q^2)_j(q^2)_k}{(q^2)_a(q^2)_b(q^2)_c}
\Rm_{a,b,c}^{i,j,k},\label{rtb}
\end{align}
where $\xi = 0,1$ is specified by 
$\xi \equiv (a-j)(c-j)$ mod 2.

\begin{example}\label{ex:skrb}
The following is the list of all the nonzero $\Rm^{a,b,c}_{3,1,2}$.
\begin{alignat*}{2}
\Rm^{1,3,0}_{3,1,2} &=  -q^2 (1 - q^4) (1 - q^6),\quad &
\Rm^{2,2,1}_{3,1,2} &=  (1 + q^2) (1 - q^6) (1 - q^2 - q^6),\\
\Rm^{1,3,0}_{3,1,2} &=  q^6, &
\Rm^{3,1,2}_{3,1,2} &=  -q^2 (-1 - q^2 + q^6 + q^8 + q^{10}).
\end{alignat*}
We see 
$\lim_{q\rightarrow 0}\Rm^{a,b,c}_{3,1,2} = \delta^a_2\delta^b_2\delta^c_1$
in agreement with (\ref{mrei}) with $\epsilon=0$.

The following is the list of all the nonzero $\Rm^{a,b}_{3,1}$.
\begin{alignat*}{2}
\Rm^{1,3}_{3,1} &= ({\bf a}^-)^3 {\bf a}^+ - q^{-4}(1+q^2+q^4)({\bf a}^-)^2{\bf k}^2,
\quad &
\Rm^{4,0}_{3,1} &= {\bf a}^+ {\bf k}^3, \\
\Rm^{3,1}_{3,1} &= q^{-2}(1+q^2+q^4) {\bf a}^- {\bf a}^+{\bf k}^2
-q^{-2} {\bf k}^4, &
\Rm^{0,4}_{3,1} &= -q^{-2}({\bf a}^-)^3{\bf k}.\\
\Rm^{2,2}_{3,1} & = q^{-4}(1+q^2+q^4)\bigl(q^2({\bf a}^-)^2{\bf a}^+{\bf k}
-{\bf a}^-{\bf k}^3\bigr).
\end{alignat*}
A part of them will be used in Example \ref{ex:yum}.
\end{example}

Let us proceed to the 3D $L$ \cite{BS}.
Set $V = \C v_0 \oplus \C v_1$.
We define a three dimensional $L$ operator, 3D $L$ for short,  
${\mathscr L} \in \mathrm{End}(V^{\otimes 2}\otimes F)$ by 
a format parallel with (\ref{rop}):
\begin{align}\label{rino}
{\mathscr L}(v_\alpha \otimes v_\beta \otimes |m\rangle)
= \sum_{\gamma,\delta}v_\gamma \otimes v_\delta \otimes 
{\mathscr L}_{\alpha, \beta}^{\gamma,\delta}|m\rangle,
\end{align}
where ${\mathscr L}_{\alpha, \beta}^{\gamma,\delta}
\in \mathrm{End}(F)$ are zero except the following six cases:
\begin{align}\label{lak}
{\mathscr L}_{0, 0}^{0,0}&= {\mathscr L}_{1,1}^{1,1} = 1,\;\;
{\mathscr L}_{0, 1}^{0,1} = -q{\bf k},\;\;
{\mathscr L}_{1,0}^{1,0} = {\bf k},\;\;
{\mathscr L}_{1,0}^{0,1} = {\bf a}^-,\;\;
{\mathscr L}^{1,0}_{0,1} = {\bf a}^+.
\end{align}
Thus ${\mathscr L}$ may be regarded 
as defining a six-vertex model \cite{Bax}
whose Boltzmann weights take values in the $q$-oscillators.
One may also write (\ref{rino}) like (\ref{Rabc}) as
\begin{equation}\label{Lex}
\begin{split}
&{\mathscr L}(v_\alpha \otimes v_\beta \otimes |m\rangle) = 
\sum_{\gamma,\delta, j} {\mathscr L}^{\gamma,\delta, j}_{\alpha, \beta, m}
v_\gamma \otimes v_\delta \otimes |j\rangle,\\
&{\mathscr L}^{0, 0, j}_{0, 0, m}
={\mathscr L}^{1, 1, j}_{1, 1, m}=\delta^j_m,\quad
{\mathscr L}^{0, 1, j}_{0, 1, m}= -\delta^j_mq^{m+1},\quad
{\mathscr L}^{1, 0, j}_{1, 0, m}=\delta^j_mq^m,\\
&{\mathscr L}^{0, 1, j}_{1, 0, m}=\delta^j_{m-1}(1-q^{2m}),\quad
{\mathscr L}^{1, 0, j}_{0, 1, m}=\delta^j_{m+1}.
\end{split}
\end{equation}
The other ${\mathscr L}^{\gamma,\delta, j}_{\alpha, \beta, m}$ are zero.

We assign a solid arrow to $F$ and 
a dotted arrow to $V$, and 
depict the matrix elements of 3D $R$ and 3D $L$ as 
\[
\begin{picture}(200,49)(-90,-35)
\put(-122,-20){$\Rm^{a,b,c}_{i,j,k}=$}
\put(-30,-10){\vector(-3,-1){40}}
\put(-78,-27){$\scriptstyle{c}$}
\put(-48,-16){\vector(0,1){20}}\put(-48,-16){\line(0,-1){16}}
\put(-48,-16){\vector(3,-1){18}} \put(-48,-16){\line(-3,1){18}}
\put(-51,7){$\scriptstyle{b}$}
\put(-72,-10){$\scriptstyle{i}$}
\put(-27,-26){$\scriptstyle{a}$}
\put(-50,-39){$\scriptstyle{j}$}
\put(-27,-12){$\scriptstyle{k}$}
\put(130,-12){
\put(-88,-7){${\mathscr L}^{a,b,c}_{i,j,k}=$}
\put(3,1){\vector(-3,-1){40}}
\multiput(-15,-18)(0,3){10}{.}\put(-13.5,12.6){\vector(0,1){1}}
\multiput(-30,0)(3,-1){11}{.}\put(5,-10.6){\vector(3,-1){1}}
\put(-44,-14){$\scriptstyle{c}$}
\put(-37,0){$\scriptstyle{i}$}
\put(-16,17){$\scriptstyle{b}$}
\put(-16,-26){$\scriptstyle{j}$}
\put(9,-14){$\scriptstyle{a}$}
\put(6,0){$\scriptstyle{k}$}
}
\end{picture}
\]
We will also depict $\Rm$ and ${\mathscr L}$ by
the same diagrams with no indices. 

\section{Tetrahedron equation}\label{sec:tet}

The $\Rm$ satisfies the tetrahedron equation of $RRRR$ type \cite{KV}
\begin{align}\label{te1}
\Rm_{1,2,4}\Rm_{1,3,5}\Rm_{2,3,6}\Rm_{4,5,6}=
\Rm_{4,5,6}\Rm_{2,3,6}\Rm_{1,3,5}\Rm_{1,2,4},
\end{align}
which is an equality in $\mathrm{End}(F^{\otimes 6})$.
Here $\Rm_{i,j,k}$ acts as $\Rm$ on the 
$i,j,k$ th components from the left in the 
tensor product $F^{\otimes 6}$, and as identity elsewhere\footnote{
These indices should not be confused with 
those specifying the matrix elements $\Rm^{a,b,c}_{i,j,k}$.}.
By denoting the $F$ at the $i$ th component by a solid arrow with $i$, 
(\ref{te1}) is depicted as follows:

{\unitlength 0.1in
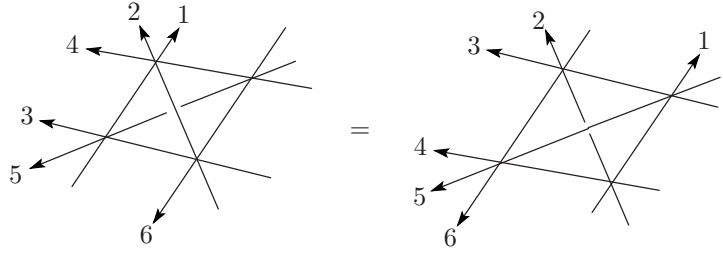
\begin{picture}(6.8900, 15)(4,-27)
%
{\color[named]{Black}{%
\special{pn 8}%
\special{pa 3938 1476}%
\special{pa 3248 2496}%
\special{fp}%
\special{sh 1}%
\special{pa 3248 2496}%
\special{pa 3302 2452}%
\special{pa 3278 2452}%
\special{pa 3270 2430}%
\special{pa 3248 2496}%
\special{fp}%
}}%
\put(32.0900,-25.6000){\makebox(0,0){6}}%
%
{\color[named]{Black}{%
\special{pn 8}%
\special{pa 4620 1722}%
\special{pa 3930 1986}%
\special{fp}%
}}%
%
{\color[named]{Black}{%
\special{pn 8}%
\special{pa 3932 1990}%
\special{pa 3112 2316}%
\special{fp}%
\special{sh 1}%
\special{pa 3112 2316}%
\special{pa 3182 2310}%
\special{pa 3162 2296}%
\special{pa 3168 2272}%
\special{pa 3112 2316}%
\special{fp}%
}}%
\put(30.4700,-23.4600){\makebox(0,0){5}}%
%
{\color[named]{Black}{%
\special{pn 8}%
\special{pa 3950 2430}%
\special{pa 4510 1606}%
\special{fp}%
\special{sh 1}%
\special{pa 4510 1606}%
\special{pa 4456 1650}%
\special{pa 4480 1650}%
\special{pa 4488 1672}%
\special{pa 4510 1606}%
\special{fp}%
}}%
\put(45.3500,-15.2700){\makebox(0,0){1}}%
%
{\color[named]{Black}{%
\special{pn 8}%
\special{pa 3912 1956}%
\special{pa 3708 1478}%
\special{fp}%
\special{sh 1}%
\special{pa 3708 1478}%
\special{pa 3716 1548}%
\special{pa 3728 1528}%
\special{pa 3752 1532}%
\special{pa 3708 1478}%
\special{fp}%
}}%
%
{\color[named]{Black}{%
\special{pn 8}%
\special{pa 3938 2018}%
\special{pa 4142 2496}%
\special{fp}%
}}%
\put(36.7100,-14.2300){\makebox(0,0){2}}%
%
{\color[named]{Black}{%
\special{pn 8}%
\special{pa 4608 1878}%
\special{pa 3404 1580}%
\special{fp}%
\special{sh 1}%
\special{pa 3404 1580}%
\special{pa 3464 1614}%
\special{pa 3456 1592}%
\special{pa 3474 1576}%
\special{pa 3404 1580}%
\special{fp}%
}}%
\put(33.3300,-15.5300){\makebox(0,0){3}}%
%
{\color[named]{Black}{%
\special{pn 8}%
\special{pa 4302 2314}%
\special{pa 3128 2106}%
\special{fp}%
\special{sh 1}%
\special{pa 3128 2106}%
\special{pa 3190 2138}%
\special{pa 3182 2116}%
\special{pa 3198 2098}%
\special{pa 3128 2106}%
\special{fp}%
}}%
\put(30.5000,-20.8500){\makebox(0,0){4}}%
%
{\color[named]{Black}{%
\special{pn 8}%
\special{pa 2268 2248}%
\special{pa 1064 1950}%
\special{fp}%
\special{sh 1}%
\special{pa 1064 1950}%
\special{pa 1124 1984}%
\special{pa 1116 1962}%
\special{pa 1134 1946}%
\special{pa 1064 1950}%
\special{fp}%
}}%
\put(9.9300,-19.2300){\makebox(0,0){3}}%
%
{\color[named]{Black}{%
\special{pn 8}%
\special{pa 2482 1780}%
\special{pa 1308 1574}%
\special{fp}%
\special{sh 1}%
\special{pa 1308 1574}%
\special{pa 1370 1604}%
\special{pa 1362 1582}%
\special{pa 1378 1566}%
\special{pa 1308 1574}%
\special{fp}%
}}%
\put(12.3000,-15.5200){\makebox(0,0){4}}%
%
{\color[named]{Black}{%
\special{pn 8}%
\special{pa 1228 2294}%
\special{pa 1786 1468}%
\special{fp}%
\special{sh 1}%
\special{pa 1786 1468}%
\special{pa 1732 1512}%
\special{pa 1756 1512}%
\special{pa 1766 1534}%
\special{pa 1786 1468}%
\special{fp}%
}}%
\put(18.1200,-13.9000){\makebox(0,0){1}}%
%
{\color[named]{Black}{%
\special{pn 8}%
\special{pa 1994 2412}%
\special{pa 1584 1456}%
\special{fp}%
\special{sh 1}%
\special{pa 1584 1456}%
\special{pa 1592 1524}%
\special{pa 1606 1504}%
\special{pa 1630 1508}%
\special{pa 1584 1456}%
\special{fp}%
}}%
\put(15.5200,-13.7700){\makebox(0,0){2}}%
%
{\color[named]{Black}{%
\special{pn 8}%
\special{pa 2398 1638}%
\special{pa 1800 1874}%
\special{fp}%
}}%
%
{\color[named]{Black}{%
\special{pn 8}%
\special{pa 1722 1908}%
\special{pa 1012 2198}%
\special{fp}%
\special{sh 1}%
\special{pa 1012 2198}%
\special{pa 1082 2190}%
\special{pa 1062 2178}%
\special{pa 1066 2154}%
\special{pa 1012 2198}%
\special{fp}%
}}%
\put(9.3400,-22.2900){\makebox(0,0){5}}%
%
{\color[named]{Black}{%
\special{pn 8}%
\special{pa 2346 1462}%
\special{pa 1656 2482}%
\special{fp}%
\special{sh 1}%
\special{pa 1656 2482}%
\special{pa 1710 2438}%
\special{pa 1686 2438}%
\special{pa 1678 2416}%
\special{pa 1656 2482}%
\special{fp}%
}}%
\put(16.1700,-25.4700){\makebox(0,0){6}}%
\put(27.3500,-20.0800){\makebox(0,0){$=$}}%
\end{picture}}%

The ${\mathscr L}$ satisfies the tetrahedron equation of $RLLL$ type 
\cite{BS}
\begin{align}\label{te2}
{\mathscr L}_{1,2,4}{\mathscr L}_{1,3,5}{\mathscr L}_{2,3,6}\Rm_{4,5,6}=
\Rm_{4,5,6}{\mathscr L}_{2,3,6}{\mathscr L}_{1,3,5}{\mathscr L}_{1,2,4},
\end{align}
which is an equality in 
$\mathrm{End}({V}^{\otimes 3}\otimes F^{\otimes 3})$.
The indices are assigned according to the same rule as in (\ref{te1}).
By denoting the $V$ at the $i$ th component by a dotted arrow with $i$, 
(\ref{te2}) is depicted as follows:

{\unitlength 0.1in
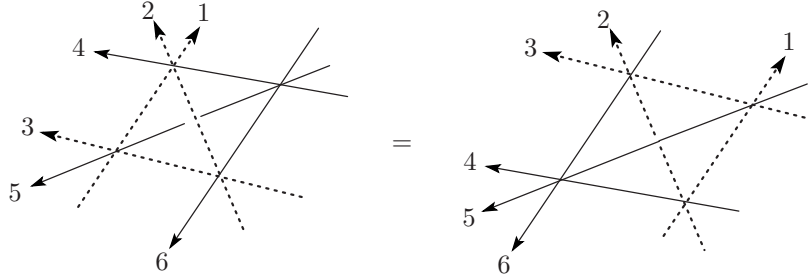
\begin{picture}(  6.8900, 16.5)(6,-28.5)
%
{\color[named]{Black}{%
\special{pn 8}%
\special{pa 4320 1504}%
\special{pa 3546 2652}%
\special{fp}%
\special{sh 1}%
\special{pa 3546 2652}%
\special{pa 3600 2608}%
\special{pa 3576 2608}%
\special{pa 3566 2586}%
\special{pa 3546 2652}%
\special{fp}%
}}%
\put(29.6700,-21.0300){\makebox(0,0){$=$}}%
\put(17.1000,-27.0900){\makebox(0,0){6}}%
%
{\color[named]{Black}{%
\special{pn 8}%
\special{pa 2530 1490}%
\special{pa 1754 2638}%
\special{fp}%
\special{sh 1}%
\special{pa 1754 2638}%
\special{pa 1808 2594}%
\special{pa 1784 2594}%
\special{pa 1776 2572}%
\special{pa 1754 2638}%
\special{fp}%
}}%
\put(9.4100,-23.5200){\makebox(0,0){5}}%
%
{\color[named]{Black}{%
\special{pn 8}%
\special{pa 1828 1990}%
\special{pa 1030 2316}%
\special{fp}%
\special{sh 1}%
\special{pa 1030 2316}%
\special{pa 1098 2310}%
\special{pa 1078 2296}%
\special{pa 1084 2272}%
\special{pa 1030 2316}%
\special{fp}%
}}%
%
{\color[named]{Black}{%
\special{pn 8}%
\special{pa 2588 1686}%
\special{pa 1914 1952}%
\special{fp}%
}}%
\put(16.3700,-13.9400){\makebox(0,0){2}}%
%
{\color[named]{Black}{%
\special{pn 13}%
\special{pa 2136 2542}%
\special{pa 1674 1466}%
\special{dt 0.045}%
\special{sh 1}%
\special{pa 1674 1466}%
\special{pa 1682 1536}%
\special{pa 1694 1516}%
\special{pa 1718 1520}%
\special{pa 1674 1466}%
\special{fp}%
}}%
\put(19.2900,-14.0800){\makebox(0,0){1}}%
%
{\color[named]{Black}{%
\special{pn 13}%
\special{pa 1272 2426}%
\special{pa 1900 1496}%
\special{dt 0.045}%
\special{sh 1}%
\special{pa 1900 1496}%
\special{pa 1846 1540}%
\special{pa 1870 1540}%
\special{pa 1880 1562}%
\special{pa 1900 1496}%
\special{fp}%
}}%
\put(12.7500,-15.9000){\makebox(0,0){4}}%
%
{\color[named]{Black}{%
\special{pn 8}%
\special{pa 2682 1848}%
\special{pa 1362 1614}%
\special{fp}%
\special{sh 1}%
\special{pa 1362 1614}%
\special{pa 1424 1646}%
\special{pa 1416 1624}%
\special{pa 1432 1606}%
\special{pa 1362 1614}%
\special{fp}%
}}%
\put(10.0800,-20.0700){\makebox(0,0){3}}%
%
{\color[named]{Black}{%
\special{pn 13}%
\special{pa 2442 2374}%
\special{pa 1088 2038}%
\special{dt 0.045}%
\special{sh 1}%
\special{pa 1088 2038}%
\special{pa 1148 2072}%
\special{pa 1140 2050}%
\special{pa 1158 2034}%
\special{pa 1088 2038}%
\special{fp}%
}}%
\put(33.2200,-21.9000){\makebox(0,0){4}}%
%
{\color[named]{Black}{%
\special{pn 8}%
\special{pa 4730 2448}%
\special{pa 3410 2214}%
\special{fp}%
\special{sh 1}%
\special{pa 3410 2214}%
\special{pa 3472 2244}%
\special{pa 3464 2222}%
\special{pa 3480 2206}%
\special{pa 3410 2214}%
\special{fp}%
}}%
\put(36.4100,-15.9200){\makebox(0,0){3}}%
%
{\color[named]{Black}{%
\special{pn 13}%
\special{pa 5074 1958}%
\special{pa 3720 1620}%
\special{dt 0.045}%
\special{sh 1}%
\special{pa 3720 1620}%
\special{pa 3780 1656}%
\special{pa 3772 1634}%
\special{pa 3790 1618}%
\special{pa 3720 1620}%
\special{fp}%
}}%
\put(40.2000,-14.4500){\makebox(0,0){2}}%
%
{\color[named]{Black}{%
\special{pn 13}%
\special{pa 4320 2114}%
\special{pa 4550 2652}%
\special{dt 0.045}%
}}%
%
{\color[named]{Black}{%
\special{pn 13}%
\special{pa 4290 2046}%
\special{pa 4062 1508}%
\special{dt 0.045}%
\special{sh 1}%
\special{pa 4062 1508}%
\special{pa 4070 1576}%
\special{pa 4082 1556}%
\special{pa 4106 1562}%
\special{pa 4062 1508}%
\special{fp}%
}}%
\put(49.9200,-15.6200){\makebox(0,0){1}}%
%
{\color[named]{Black}{%
\special{pn 13}%
\special{pa 4336 2578}%
\special{pa 4964 1650}%
\special{dt 0.045}%
\special{sh 1}%
\special{pa 4964 1650}%
\special{pa 4910 1694}%
\special{pa 4934 1694}%
\special{pa 4944 1716}%
\special{pa 4964 1650}%
\special{fp}%
}}%
\put(33.1800,-24.8300){\makebox(0,0){5}}%
%
{\color[named]{Black}{%
\special{pn 8}%
\special{pa 4314 2082}%
\special{pa 3392 2448}%
\special{fp}%
\special{sh 1}%
\special{pa 3392 2448}%
\special{pa 3462 2442}%
\special{pa 3442 2428}%
\special{pa 3448 2406}%
\special{pa 3392 2448}%
\special{fp}%
}}%
%
{\color[named]{Black}{%
\special{pn 8}%
\special{pa 5088 1782}%
\special{pa 4312 2080}%
\special{fp}%
}}%
\put(35.0100,-27.2400){\makebox(0,0){6}}%
\end{picture}}%

Viewed as an equation on ${\mathscr R}$, 
(\ref{te2}) is equivalent to the 
intertwining relation for the irreducible representations of the 
quantized coordinate ring $A_q(sl_3)$ \cite{KV}, \cite[eq.(2.15)]{KO1}  
in the sense that the both lead to the same solution given in 
(\ref{Rex1})--(\ref{Rex3})
up to an overall  normalization.

One can concatenate the tetrahedron equations 
to form the $n$-layer versions mixing the two types 
(\ref{te1}) and (\ref{te2}) arbitrarily. 
To describe them we introduce the notation unifying 
$F, V$ and $\Rm, {\mathscr L}$.
\begin{align}
W^{(\epsilon)}= \begin{cases}
F,\\ V,
\end{cases}
{\mathscr S}^{(\epsilon)} = \begin{cases}
\Rm, \\ {\mathscr L},
\end{cases}
{\mathscr S}^{(\epsilon)\, a,b}_{\phantom{(\epsilon)}\, i,j}
=\begin{cases}
\Rm^{a,b}_{i,j}, \\ {\mathscr L}^{a,b}_{i,j},
\end{cases}
{\mathscr S}^{(\epsilon)\, a,b,c}_{\phantom{(\epsilon)}\, i,j,k}
=\begin{cases}
\Rm^{a,b,c}_{i,j,k} &\;(\epsilon=0),
\\ {\mathscr L}^{a,b,c}_{i,j,k}&\;(\epsilon=1).
\end{cases}\label{misto}
\end{align}
Note that
\begin{align}\label{Lins}
{\mathscr S}^{(\epsilon)\, a, b, c}_{\phantom{(\epsilon)}\,i, j, k} = 0\;\;
\text{unless}\;\; (a+b,b+c)=(i+j,j+k).
\end{align}
Now (\ref{te1}) and (\ref{te2}) are written as
\begin{align}\label{te3}
{\mathscr S}^{(\epsilon)}_{1,2,4}{\mathscr S}^{(\epsilon)}_{1,3,5}{\mathscr S}^{(\epsilon)}_{2,3,6}\Rm_{4,5,6}=
\Rm_{4,5,6}{\mathscr S}^{(\epsilon)}_{2,3,6}{\mathscr S}^{(\epsilon)}_{1,3,5}{\mathscr S}^{(\epsilon)}_{1,2,4}
\quad (\epsilon=0,1)
\end{align}
which is an equality in 
$\mathrm{End}(
W^{(\epsilon)} \otimes W^{(\epsilon)}\otimes W^{(\epsilon)}
\otimes F\otimes F \otimes F)$.

Let $n$ be a positive integer.
Given an arbitrary sequence 
$(\epsilon_1,\ldots, \epsilon_n) \in \{0,1\}^n$, we set 
\begin{align}\label{Wdef}
&\W= W^{(\epsilon_1)} \otimes \cdots \otimes W^{(\epsilon_n)}.
\end{align}
Let $\overset{\alpha_i}{W^{(\epsilon_i)}},
\overset{\beta_i}{W^{(\epsilon_i)}},
\overset{\gamma_i}{W^{(\epsilon_i)}}$ be copies of $W^{(\epsilon_i)}$,
where $\alpha_i, \beta_i$ and $\gamma_i\,(i=1,\ldots, n)$
are distinct labels.
Replacing the spaces $1,2,3$ by them in (\ref{te3})
we have 
\begin{align*}
{\mathscr S}^{(\epsilon_i)}_{\alpha_i, \beta_i, 4}
{\mathscr S}^{(\epsilon_i)}_{\alpha_i, \gamma_i, 5}
{\mathscr S}^{(\epsilon_i)}_{\beta_i, \gamma_i, 6}\Rm_{4,5,6}=
\Rm_{4,5,6}
{\mathscr S}^{(\epsilon_i)}_{\beta_i, \gamma_i, 6}
{\mathscr S}^{(\epsilon_i)}_{\alpha_i, \gamma_i, 5}
{\mathscr S}^{(\epsilon_i)}_{\alpha_i, \beta_i, 4}
\end{align*}
for each $i$.
Thus for any $i$ one can let $\Rm_{4,5,6}$ penetrate   
${\mathscr S}^{(\epsilon_i)}_{\alpha_i, \beta_i, 4}
{\mathscr S}^{(\epsilon_i)}_{\alpha_i, \gamma_i, 5}
{\mathscr S}^{(\epsilon_i)}_{\beta_i, \gamma_i, 6}$ to the left 
transforming it into the reverse order product
${\mathscr S}^{(\epsilon_i)}_{\beta_i, \gamma_i, 6}
{\mathscr S}^{(\epsilon_i)}_{\alpha_i, \gamma_i, 5}
{\mathscr S}^{(\epsilon_i)}_{\alpha_i, \beta_i, 4}$.
Repeating this $n$ times leads to
\begin{equation}\label{TEn}
\begin{split}
&\bigl({\mathscr S}^{(\epsilon_1)}_{\alpha_1, \beta_1, 4}
{\mathscr S}^{(\epsilon_1)}_{\alpha_1, \gamma_1, 5}
{\mathscr S}^{(\epsilon_1)}_{\beta_1, \gamma_1, 6}\bigr)
\cdots 
\bigl({\mathscr S}^{(\epsilon_n)}_{\alpha_n, \beta_n, 4}
{\mathscr S}^{(\epsilon_n)}_{\alpha_n, \gamma_n, 5}
{\mathscr S}^{(\epsilon_n)}_{\beta_n, \gamma_n, 6}\bigr)\Rm_{4,5,6}\\
&= 
\Rm_{4,5,6}
\bigl({\mathscr S}^{(\epsilon_1)}_{\beta_1, \gamma_1, 6}
{\mathscr S}^{(\epsilon_1)}_{\alpha_1, \gamma_1, 5}
{\mathscr S}^{(\epsilon_1)}_{\alpha_1, \beta_1, 4}
\bigr)
\cdots 
\bigl({\mathscr S}^{(\epsilon_n)}_{\beta_n, \gamma_n, 6}
{\mathscr S}^{(\epsilon_n)}_{\alpha_n, \gamma_n, 5}
{\mathscr S}^{(\epsilon_n)}_{\alpha_n, \beta_n, 4}
\bigr).
\end{split}
\end{equation}
This is an equality in 
$\mathrm{End}(\overset{\boldsymbol \alpha}{\W}\otimes 
\overset{\boldsymbol\beta}{\W}\otimes 
\overset{\boldsymbol\gamma}{\W}\otimes 
\overset{4}{F}\otimes 
\overset{5}{F}\otimes 
\overset{6}{F})$,
where 
${\boldsymbol\alpha}=(\alpha_1,\ldots, \alpha_n)$
is the array of labels and 
$\overset{\boldsymbol \alpha}{\W}= 
\overset{\alpha_1}{W^{(\epsilon_1)}}\otimes \cdots \otimes 
\overset{\alpha_n}{W^{(\epsilon_n)}}$.
The spaces
$\overset{\boldsymbol\beta}{\W}$ and $\overset{\boldsymbol\gamma}{\W}$
should be understood similarly.
They are just copies of $\W$ in (\ref{Wdef}).
The relation (\ref{TEn}) is depicted as follows:

{\unitlength 0.1in
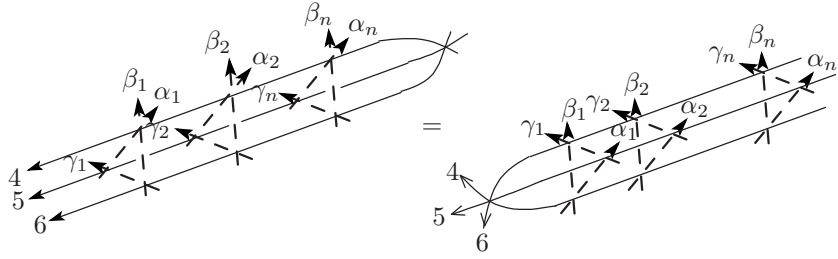
\begin{picture}(14.1700,  14)(1,-35.7700)
%
{\color[named]{Black}{%
\special{pn 8}%
\special{pa 3028 3078}%
\special{pa 4446 2562}%
\special{fp}%
}}%
\put(25.1300,-29.2900){\makebox(0,0){$=$}}%
\put(3.4900,-33.0900){\makebox(0,0){5}}%
%
{\color[named]{Black}{%
\special{pn 8}%
\special{pa 2390 2582}%
\special{pa 2422 2572}%
\special{pa 2452 2564}%
\special{pa 2482 2552}%
\special{pa 2512 2542}%
\special{pa 2540 2528}%
\special{pa 2570 2514}%
\special{pa 2598 2498}%
\special{pa 2624 2482}%
\special{pa 2680 2446}%
\special{pa 2700 2434}%
\special{fp}%
}}%
%
{\color[named]{Black}{%
\special{pn 8}%
\special{pa 2352 2736}%
\special{pa 2384 2726}%
\special{pa 2416 2714}%
\special{pa 2446 2700}%
\special{pa 2474 2684}%
\special{pa 2498 2666}%
\special{pa 2520 2644}%
\special{pa 2536 2618}%
\special{pa 2552 2590}%
\special{pa 2564 2558}%
\special{pa 2586 2496}%
\special{pa 2596 2466}%
\special{pa 2608 2438}%
\special{pa 2612 2424}%
\special{fp}%
}}%
%
{\color[named]{Black}{%
\special{pn 8}%
\special{pa 2236 2472}%
\special{pa 2300 2464}%
\special{pa 2364 2460}%
\special{pa 2428 2460}%
\special{pa 2460 2464}%
\special{pa 2492 2470}%
\special{pa 2522 2478}%
\special{pa 2582 2502}%
\special{pa 2610 2516}%
\special{pa 2640 2532}%
\special{pa 2648 2536}%
\special{fp}%
}}%
\put(19.9800,-23.9400){\makebox(0,0)[rb]{$\beta_n$}}%
\put(10.3100,-27.4200){\makebox(0,0)[rb]{$\beta_1$}}%
\put(14.6300,-25.4300){\makebox(0,0)[rb]{$\beta_2$}}%
\put(10.5800,-28.0000){\makebox(0,0)[lb]{$\alpha_1$}}%
%
{\color[named]{Black}{%
\special{pn 8}%
\special{pa 1954 2736}%
\special{pa 1518 2894}%
\special{fp}%
}}%
%
{\color[named]{Black}{%
\special{pn 8}%
\special{pa 1450 2916}%
\special{pa 1038 3068}%
\special{fp}%
}}%
%
{\color[named]{Black}{%
\special{pn 8}%
\special{pa 2388 2588}%
\special{pa 2038 2716}%
\special{fp}%
}}%
%
{\color[named]{Black}{%
\special{pn 8}%
\special{pa 962 3096}%
\special{pa 410 3298}%
\special{fp}%
\special{sh 1}%
\special{pa 410 3298}%
\special{pa 480 3294}%
\special{pa 460 3280}%
\special{pa 466 3256}%
\special{pa 410 3298}%
\special{fp}%
}}%
%
{\color[named]{Black}{%
\special{pn 8}%
\special{pa 2352 2742}%
\special{pa 516 3412}%
\special{fp}%
\special{sh 1}%
\special{pa 516 3412}%
\special{pa 586 3408}%
\special{pa 566 3394}%
\special{pa 572 3370}%
\special{pa 516 3412}%
\special{fp}%
}}%
%
{\color[named]{Black}{%
\special{pn 8}%
\special{pa 2236 2478}%
\special{pa 402 3148}%
\special{fp}%
\special{sh 1}%
\special{pa 402 3148}%
\special{pa 470 3144}%
\special{pa 452 3130}%
\special{pa 458 3106}%
\special{pa 402 3148}%
\special{fp}%
}}%
%
{\color[named]{Black}{%
\special{pn 13}%
\special{pa 2082 2890}%
\special{pa 1726 2750}%
\special{da 0.070}%
\special{sh 1}%
\special{pa 1726 2750}%
\special{pa 1782 2792}%
\special{pa 1776 2770}%
\special{pa 1796 2756}%
\special{pa 1726 2750}%
\special{fp}%
}}%
%
{\color[named]{Black}{%
\special{pn 13}%
\special{pa 2014 2918}%
\special{pa 1974 2424}%
\special{da 0.070}%
\special{sh 1}%
\special{pa 1974 2424}%
\special{pa 1960 2492}%
\special{pa 1978 2478}%
\special{pa 2000 2490}%
\special{pa 1974 2424}%
\special{fp}%
}}%
%
{\color[named]{Black}{%
\special{pn 13}%
\special{pa 1772 2826}%
\special{pa 2068 2466}%
\special{da 0.070}%
\special{sh 1}%
\special{pa 2068 2466}%
\special{pa 2010 2504}%
\special{pa 2034 2506}%
\special{pa 2042 2530}%
\special{pa 2068 2466}%
\special{fp}%
}}%
%
{\color[named]{Black}{%
\special{pn 13}%
\special{pa 1572 3090}%
\special{pa 1182 2936}%
\special{da 0.070}%
\special{sh 1}%
\special{pa 1182 2936}%
\special{pa 1236 2980}%
\special{pa 1232 2956}%
\special{pa 1250 2942}%
\special{pa 1182 2936}%
\special{fp}%
}}%
%
{\color[named]{Black}{%
\special{pn 13}%
\special{pa 1498 3120}%
\special{pa 1454 2578}%
\special{da 0.070}%
\special{sh 1}%
\special{pa 1454 2578}%
\special{pa 1440 2646}%
\special{pa 1458 2630}%
\special{pa 1480 2642}%
\special{pa 1454 2578}%
\special{fp}%
}}%
%
{\color[named]{Black}{%
\special{pn 13}%
\special{pa 1232 3020}%
\special{pa 1558 2624}%
\special{da 0.070}%
\special{sh 1}%
\special{pa 1558 2624}%
\special{pa 1500 2662}%
\special{pa 1524 2664}%
\special{pa 1530 2688}%
\special{pa 1558 2624}%
\special{fp}%
}}%
%
{\color[named]{Black}{%
\special{pn 13}%
\special{pa 1084 3252}%
\special{pa 728 3110}%
\special{da 0.070}%
\special{sh 1}%
\special{pa 728 3110}%
\special{pa 784 3154}%
\special{pa 778 3130}%
\special{pa 798 3116}%
\special{pa 728 3110}%
\special{fp}%
}}%
%
{\color[named]{Black}{%
\special{pn 13}%
\special{pa 1016 3280}%
\special{pa 976 2786}%
\special{da 0.070}%
\special{sh 1}%
\special{pa 976 2786}%
\special{pa 962 2854}%
\special{pa 980 2838}%
\special{pa 1002 2850}%
\special{pa 976 2786}%
\special{fp}%
}}%
%
{\color[named]{Black}{%
\special{pn 13}%
\special{pa 774 3188}%
\special{pa 1070 2826}%
\special{da 0.070}%
\special{sh 1}%
\special{pa 1070 2826}%
\special{pa 1012 2866}%
\special{pa 1036 2868}%
\special{pa 1044 2890}%
\special{pa 1070 2826}%
\special{fp}%
}}%
\put(4.6500,-34.3200){\makebox(0,0){6}}%
\put(3.3000,-31.8700){\makebox(0,0){4}}%
\put(17.1400,-27.9400){\makebox(0,0)[rb]{$\gamma_n$}}%
\put(20.7500,-24.4000){\makebox(0,0)[lb]{$\alpha_n$}}%
\put(11.4700,-29.8700){\makebox(0,0)[rb]{$\gamma_2$}}%
\put(15.7300,-26.1300){\makebox(0,0)[lb]{$\alpha_2$}}%
\put(7.1600,-31.6000){\makebox(0,0)[rb]{$\gamma_1$}}%
\put(27.7700,-35.2800){\makebox(0,0){6}}%
\put(26.2200,-31.5400){\makebox(0,0){4}}%
\put(25.5100,-33.9200){\makebox(0,0){5}}%
\put(41.1000,-25.8800){\makebox(0,0)[rb]{$\gamma_n$}}%
\put(41.6800,-25.1000){\makebox(0,0)[lb]{$\beta_n$}}%
\put(44.7700,-26.2600){\makebox(0,0)[lb]{$\alpha_n$}}%
\put(34.5300,-28.4500){\makebox(0,0)[rb]{$\gamma_2$}}%
\put(38.1400,-28.7100){\makebox(0,0)[lb]{$\alpha_2$}}%
\put(35.1800,-27.6100){\makebox(0,0)[lb]{$\beta_2$}}%
\put(31.1200,-29.7400){\makebox(0,0)[rb]{$\gamma_1$}}%
\put(31.8900,-28.9600){\makebox(0,0)[lb]{$\beta_1$}}%
\put(34.4600,-29.9300){\makebox(0,0)[lb]{$\alpha_1$}}%
%
{\color[named]{Black}{%
\special{pn 8}%
\special{pa 2668 3336}%
\special{pa 2616 3380}%
\special{fp}%
\special{pa 2616 3380}%
\special{pa 2686 3384}%
\special{fp}%
}}%
%
{\color[named]{Black}{%
\special{pn 8}%
\special{pa 2770 3374}%
\special{pa 2790 3440}%
\special{fp}%
\special{pa 2790 3440}%
\special{pa 2822 3378}%
\special{fp}%
}}%
%
{\color[named]{Black}{%
\special{pn 8}%
\special{pa 2726 3206}%
\special{pa 2662 3176}%
\special{fp}%
\special{pa 2662 3176}%
\special{pa 2688 3240}%
\special{fp}%
}}%
%
{\color[named]{Black}{%
\special{pn 8}%
\special{pa 3086 3206}%
\special{pa 3056 3218}%
\special{pa 3028 3230}%
\special{pa 2938 3266}%
\special{pa 2908 3276}%
\special{pa 2878 3288}%
\special{pa 2788 3322}%
\special{pa 2698 3352}%
\special{pa 2666 3364}%
\special{pa 2636 3374}%
\special{pa 2616 3380}%
\special{fp}%
}}%
%
{\color[named]{Black}{%
\special{pn 8}%
\special{pa 3164 3336}%
\special{pa 3132 3340}%
\special{pa 3068 3350}%
\special{pa 3036 3354}%
\special{pa 2972 3354}%
\special{pa 2940 3352}%
\special{pa 2908 3346}%
\special{pa 2878 3338}%
\special{pa 2848 3328}%
\special{pa 2820 3314}%
\special{pa 2794 3296}%
\special{pa 2768 3276}%
\special{pa 2744 3256}%
\special{pa 2720 3234}%
\special{pa 2672 3186}%
\special{pa 2668 3180}%
\special{fp}%
}}%
%
{\color[named]{Black}{%
\special{pn 8}%
\special{pa 3022 3078}%
\special{pa 2966 3110}%
\special{pa 2940 3128}%
\special{pa 2914 3148}%
\special{pa 2892 3170}%
\special{pa 2872 3192}%
\special{pa 2856 3218}%
\special{pa 2840 3246}%
\special{pa 2828 3276}%
\special{pa 2818 3306}%
\special{pa 2808 3338}%
\special{pa 2800 3370}%
\special{pa 2786 3436}%
\special{pa 2784 3452}%
\special{fp}%
}}%
%
{\color[named]{Black}{%
\special{pn 13}%
\special{pa 4206 3006}%
\special{pa 4502 2646}%
\special{da 0.070}%
\special{sh 1}%
\special{pa 4502 2646}%
\special{pa 4444 2684}%
\special{pa 4468 2686}%
\special{pa 4476 2710}%
\special{pa 4502 2646}%
\special{fp}%
}}%
%
{\color[named]{Black}{%
\special{pn 13}%
\special{pa 4276 3008}%
\special{pa 4240 2544}%
\special{da 0.070}%
\special{sh 1}%
\special{pa 4240 2544}%
\special{pa 4224 2612}%
\special{pa 4244 2596}%
\special{pa 4264 2608}%
\special{pa 4240 2544}%
\special{fp}%
}}%
%
{\color[named]{Black}{%
\special{pn 13}%
\special{pa 4510 2742}%
\special{pa 4136 2594}%
\special{da 0.070}%
\special{sh 1}%
\special{pa 4136 2594}%
\special{pa 4192 2638}%
\special{pa 4186 2614}%
\special{pa 4206 2600}%
\special{pa 4136 2594}%
\special{fp}%
}}%
%
{\color[named]{Black}{%
\special{pn 13}%
\special{pa 3544 3252}%
\special{pa 3840 2890}%
\special{da 0.070}%
\special{sh 1}%
\special{pa 3840 2890}%
\special{pa 3782 2930}%
\special{pa 3806 2932}%
\special{pa 3812 2954}%
\special{pa 3840 2890}%
\special{fp}%
}}%
%
{\color[named]{Black}{%
\special{pn 13}%
\special{pa 3614 3252}%
\special{pa 3576 2788}%
\special{da 0.070}%
\special{sh 1}%
\special{pa 3576 2788}%
\special{pa 3562 2856}%
\special{pa 3580 2842}%
\special{pa 3602 2854}%
\special{pa 3576 2788}%
\special{fp}%
}}%
%
{\color[named]{Black}{%
\special{pn 13}%
\special{pa 3846 2988}%
\special{pa 3472 2840}%
\special{da 0.070}%
\special{sh 1}%
\special{pa 3472 2840}%
\special{pa 3528 2882}%
\special{pa 3522 2860}%
\special{pa 3542 2846}%
\special{pa 3472 2840}%
\special{fp}%
}}%
%
{\color[named]{Black}{%
\special{pn 13}%
\special{pa 3190 3374}%
\special{pa 3486 3012}%
\special{da 0.070}%
\special{sh 1}%
\special{pa 3486 3012}%
\special{pa 3428 3052}%
\special{pa 3452 3054}%
\special{pa 3458 3076}%
\special{pa 3486 3012}%
\special{fp}%
}}%
%
{\color[named]{Black}{%
\special{pn 13}%
\special{pa 3260 3376}%
\special{pa 3222 2910}%
\special{da 0.070}%
\special{sh 1}%
\special{pa 3222 2910}%
\special{pa 3208 2978}%
\special{pa 3226 2962}%
\special{pa 3248 2974}%
\special{pa 3222 2910}%
\special{fp}%
}}%
%
{\color[named]{Black}{%
\special{pn 13}%
\special{pa 3492 3110}%
\special{pa 3118 2962}%
\special{da 0.070}%
\special{sh 1}%
\special{pa 3118 2962}%
\special{pa 3174 3004}%
\special{pa 3168 2982}%
\special{pa 3188 2968}%
\special{pa 3118 2962}%
\special{fp}%
}}%
%
{\color[named]{Black}{%
\special{pn 8}%
\special{pa 3086 3206}%
\special{pa 4626 2648}%
\special{fp}%
}}%
%
{\color[named]{Black}{%
\special{pn 8}%
\special{pa 3158 3336}%
\special{pa 4574 2820}%
\special{fp}%
}}%
\end{picture}}%

\noindent
Here the broken arrows represent either 
solid or dotted arrows depending on whether 
the corresponding $\epsilon_i$ is 0 or 1.
The vertices on the $i$ th layer 
${\mathscr S}^{(\epsilon_i)}$
should also be understood as ${\mathscr R}$ or ${\mathscr L}$
accordingly.

\section{Solution to the Yang-Baxter equation}\label{sec:ybe}
One can reduce (\ref{TEn}) to the Yang-Baxter equation
involving spectral parameters.
In this paper we shall only consider the reduction by trace.
See \cite{KS, KOS} for another reduction by using boundary vectors.

Define ${\bf h} \in \mathrm{End}(F)$
by ${\bf h}|m\rangle = m |m\rangle$. 
By (\ref{Lins}), 
$[x^{{\bf h}_4+{\bf h}_5} y^{{\bf h}_5+{\bf h}_6}, \Rm_{4,5,6}]=0$ holds
for parameters $x$ and $y$,
where the indices specify the spaces on which the operators act nontrivially.
Multiply $\Rm^{-1}_{4,5,6} x^{{\bf h}_4+{\bf h}_5} y^{{\bf h}_5+{\bf h}_6}
= x^{{\bf h}_4+{\bf h}_5} y^{{\bf h}_5+{\bf h}_6}\Rm^{-1}_{4,5,6} $ 
from the left to (\ref{TEn}) and take the trace over the space 
$F^{\otimes 3}$ corresponding to $4,5,6$.
The result becomes the Yang-Baxter equation
\begin{align}\label{sybe}
S_{\boldsymbol{\alpha, \beta}}(x)
S_{\boldsymbol{\alpha, \gamma}}(xy)
S_{\boldsymbol{\beta,\gamma}}(y)
=
S_{\boldsymbol{\beta,\gamma}}(y)
S_{\boldsymbol{\alpha, \gamma}}(xy)
S_{\boldsymbol{\alpha, \beta}}(x)
\in \mathrm{End}(
\overset{\boldsymbol\alpha}{\W}\otimes
\overset{\boldsymbol\beta}{\W}\otimes
\overset{\boldsymbol\gamma}{\W})
\end{align}
for the matrix 
$S_{\boldsymbol{\alpha, \beta}}(z)
\in \mathrm{End}
(\overset{\boldsymbol\alpha}{\W}\otimes \overset{\boldsymbol\beta}{\W})$
constructed as 
\begin{align}
S_{\boldsymbol{\alpha, \beta}}(z)
= \mathrm{Tr}_3\left(z^{{\bf h}_3}
{\mathscr S}^{(\epsilon_1)}_{\alpha_1, \beta_1, 3}
\cdots
{\mathscr S}^{(\epsilon_n)}_{\alpha_n, \beta_n, 3}\right),
\label{sdef0}
\end{align}
where $3$ denotes a copy of $F$. 
To describe the matrix elements 
of $S_{\boldsymbol{\alpha, \beta}}(z)$
we write the basis of (\ref{Wdef}) as
\begin{align}
&\W = \bigoplus_{m_1,\ldots, m_n} \!\!\C
|m_1,\ldots, m_n\rangle, \qquad
|m_1,\ldots, m_n\rangle = |m_1\rangle^{(\epsilon_1)}
\otimes \cdots \otimes 
|m_n\rangle^{(\epsilon_n)},\label{syk}\\
&|m\rangle^{(0)} = |m\rangle \in F\;\;(m \in \Z_{\ge 0}),\qquad
|m\rangle^{(1)} = v_m \in V\; \; (m\in \{0,1\}).\label{kby}
\end{align}
The range of the indices $m_i$ are to be understood as 
$\Z_{\ge 0}$ or $\{0,1\}$ according to $\epsilon_i=0$ or $1$ as in 
(\ref{kby}). 
It will crudely be denoted by $0 \le m_i \le 1/\epsilon_i$.
We use the shorthand 
$|{\bf m}\rangle = |m_1,\ldots, m_n\rangle$
for ${\bf m}=(m_1,\ldots, m_n)$ and 
write (\ref{syk}) as $\W= \bigoplus_{\bf m} \C |{\bf m}\rangle$.
We set $|{\bf m}| = m_1+\cdots + m_n$.

Let $S(z) \in \mathrm{End}(\W \otimes \W)$
denote the solution (\ref{sdef0}) 
of the Yang-Baxter equation, 
where the inessential labels 
$\boldsymbol{\alpha}, \boldsymbol{\beta}$
are now suppressed\footnote{
The labels $\boldsymbol{\alpha, \beta}\ldots$ 
introduced for the exposition of (\ref{sybe})
will no longer be used in the rest of the paper, and 
should not be confused with the indices of
$S_{l,m}(z)$ in (\ref{wl}).}.
Remember, however,  that $S(z)$ depends on 
the choice $(\epsilon_1,\ldots, \epsilon_n) \in \{0,1\}^n$.
We write its action as
\begin{align}
&S(z)\bigl(|{\bf i}\rangle \otimes |{\bf j}\rangle\bigr)
= \sum_{{\bf a},{\bf b}}
S(z)^{{\bf a},{\bf b}}_{{\bf i},{\bf j}}
|{\bf a}\rangle \otimes |{\bf b}\rangle.\label{sact}
\end{align}
Then the matrix elements are given by
\begin{align}
S(z)^{{\bf a},{\bf b}}_{{\bf i},{\bf j}}
&= \mathrm{Tr}_F\left( z^{{\bf h}}\,
{\mathscr S}^{(\epsilon_1) a_1, b_1}_{
\phantom{(\epsilon_1)}\,i_1, j_1}
\cdots 
{\mathscr S}^{(\epsilon_n) a_n, b_n}_{
\phantom{(\epsilon_1)}\,i_n, j_n}
\right)\label{sss}\\
&= \sum_{c_0, \ldots, c_{n-1}}\!\!
z^{c_0}
{\mathscr S}^{(\epsilon_1)\, a_1, b_1, c_0}_{
\phantom{(\epsilon_1)}\,i_1, j_1, c_1}
{\mathscr S}^{(\epsilon_2)\, a_2, b_2, c_1}_{
\phantom{(\epsilon_2)}\,i_2, j_2, c_2}
\cdots
{\mathscr S}^{(\epsilon_n)\, a_n, b_n, c_{n\!-\!1}}_{
\phantom{(\epsilon_n)}i_n, j_n, c_0}.\label{sabij0}
\end{align} 
The operators in (\ref{sss}) are defined by (\ref{misto}), 
(\ref{lak}) and (\ref{ropp}).
From (\ref{Lins}) it follows that 
\begin{align}
S(z)^{{\bf a},{\bf b}}_{{\bf i},{\bf j}}= 0 \;\;
\text{unless}\;\; \;{\bf a} + {\bf b} = {\bf i} + {\bf j} \;\;
\text{and}\;\;  |{\bf a}| = |{\bf i}|,\;
|{\bf b}| = |{\bf j}|.\label{iceA}
\end{align}
Given such 
${\bf a},{\bf b},{\bf i}$ and ${\bf j}$, (\ref{Lins}) further reduces 
the sums over $c_i \in \Z_{\ge 0}$ in (\ref{sabij0}) 
effectively into a {\em single} sum.
The latter property in (\ref{iceA}) implies the 
direct sum decomposition:
\begin{align}
S(z) = \bigoplus_{l,m \ge 0}  S_{l,m}(z),\;
S_{l,m}(z) \in \mathrm{End}(\W_l \otimes \W_m),\;
\W_l = \bigoplus_{{\bf m}, |{\bf m}|=l}
\!\!\!\C|{\bf m}\rangle \subset \W,
\label{wl}
\end{align}
where the former sum ranges over $0 \le l, m\le n$ if 
$\epsilon_1\cdots \epsilon_n = 1$ and 
$l,m \in \Z_{\ge 0}$ otherwise. 
The formula (\ref{sss}) is depicted as
\[
\begin{picture}(200,80)(-130,-40)

\put(3,1){\vector(-3,-1){82}}
\put(-70,-25){$\bullet$}
\put(-75,-35){$z^{{\bf h}}$}

\put(-160,-10){$S(z)^{{\bf a},{\bf b}}_{{\bf i},{\bf j}} = $}

\put(-105,-27){$\mathrm{Tr}_F\Biggl($}
\put(69,19){$\Biggr).$}

\multiput(-48,-32)(0,6){6}{\put(0,0){\line(0,1){4}}}
\multiput(-31,-22)(-6,2){6}{\put(0,0){\line(-3,1){3}}}

\put(-51,6){$\scriptstyle{b_1}$}\put(-48,3){\vector(0,1){1}}
\put(-74,-10){$\scriptstyle{i_1}$}
\put(-30,-29){$\scriptstyle{a_1}$}\put(-29,-23){\vector(3,-1){1}}
\put(-50,-39){$\scriptstyle{j_1}$}

\multiput(-15,-18)(0,6){5}{\put(0,0){\line(0,1){4}}}
\multiput(1,-10)(-6,2){5}{\put(0,0){\line(-3,1){3}}}

\put(-36,0){$\scriptstyle{i_2}$}
\put(-18,17){$\scriptstyle{b_2}$}\put(-15,13){\vector(0,1){1}}
\put(-17,-25){$\scriptstyle{j_2}$}
\put(4,-16){$\scriptstyle{a_2}$}\put(4,-11){\vector(3,-1){1}}

\multiput(5.1,1.7)(3,1){7}{.} 
\put(6,2){
\put(21,7){\line(3,1){33}}
\put(58,17){$\scriptstyle{F}$}

\multiput(35,1)(0,6){4}{\put(0,0){\line(0,1){4}}}
\multiput(46,8)(-6,2){4}{\put(0,0){\line(-3,1){3}}}

\put(15,16){$\scriptstyle{i_n}$}
\put(51,3){$\scriptstyle{a_n}$}\put(49,7){\vector(3,-1){1}}
\put(33,29){$\scriptstyle{b_n}$}\put(35,26){\vector(0,1){1}}
\put(34,-7){$\scriptstyle{j_n}$}

}
\end{picture}
\]
Here the broken arrows represent either 
solid or dotted arrows according to 
$\epsilon_i = 0$ or $1$ at the corresponding site.
Thus (\ref{sss}) is a matrix product construction of 
$S(z)$ in terms of 3D $R$ and 3D $L$ with the auxiliary space $F$.

\begin{example}\label{ex:yum}
Take $n=3$ and $(\epsilon_1,\epsilon_2,\epsilon_3)=(1,0,1)$. 
Then one has
\begin{align*}
S(z)(|031\rangle \otimes |110\rangle) &=
S^{031,110}_{031,110}(z) |031\rangle \otimes |110\rangle
+S^{040,101}_{031,110}(z) |040\rangle \otimes |101\rangle\\
&+S^{121,020}_{031,110}(z) |121\rangle \otimes |020\rangle
+S^{130,011}_{031,110}(z) |130\rangle \otimes |011\rangle,
\end{align*}
where the matrix elements are expressed as
\begin{align*}
S^{031,110}_{031,110}(z)&= \mathrm{Tr}(z^{\bf h}
{\mathscr L}^{0,1}_{0,1}\Rm^{3,1}_{3,1}{\mathscr L}^{1,0}_{1,0}),
\quad
S^{040,101}_{031,110}(z)= \mathrm{Tr}(z^{\bf h}
{\mathscr L}^{0,1}_{0,1}\Rm^{4,0}_{3,1}{\mathscr L}^{0,1}_{1,0}),\\
S^{121,020}_{031,110}(z)&= \mathrm{Tr}(z^{\bf h}
{\mathscr L}^{1,0}_{0,1}\Rm^{2,2}_{3,1}{\mathscr L}^{1,0}_{1,0}),
\quad
S^{130,011}_{031,110}(z)= \mathrm{Tr}(z^{\bf h}
{\mathscr L}^{1,0}_{0,1}\Rm^{3,1}_{3,1}{\mathscr L}^{0,1}_{1,0}).
\end{align*}
Using ${\mathscr L}^{a,b}_{i,j}$  (\ref{lak}) and 
$\Rm^{a,b}_{i,j}$ in Example \ref{ex:skrb}, 
one calculates them for instance as
\begin{align*}
S^{040,101}_{031,110}(z)&= \mathrm{Tr}\bigl(z^{\bf h}
(-q{\bf k}){\bf a}^+ {\bf k}^3{\bf a}^-\bigr)
= -q^{-2}\mathrm{Tr}\bigl({\bf k}^4z^{\bf h}{\bf a}^+ {\bf a}^-\bigr)\\
&= -q^{-2}\sum_{m \ge 0}(q^4z)^m(1-q^{2m})
=\frac{-q^2(1-q^2)z}{(1-q^4z)(1-q^6z)}.
\end{align*}
Similar calculations lead to 
\begin{align*}
S^{031,110}_{031,110}(z)&=
\frac{q^3(q^2-z)}{(1-q^4z)(1-q^6z)},\qquad
S^{121,020}_{031,110}(z)=
\frac{-q^2 (1 - q^6) (q^2 - z) z}{(1 - q^2 z) (1 - q^4 z) (1 - q^6 z)},\\
S^{130,011}_{031,110}(z)&=
\frac{-(1 - q^2) z (q^4 - z - q^2 z + q^8 z)}
{(1 - q^2 z) (1 - q^4 z) (1 - q^6 z)}.
\end{align*}
In general $S(z)^{{\bf a},{\bf b}}_{{\bf i},{\bf j}}$ is a {\em rational} function of 
$q$ and $z$. 
\end{example}

\begin{example}\label{ex:mho}
For $0\le a,b,i,j \le 1$, 
$\Rm^{a,b}_{i,j}$ (\ref{ropp}) and ${\mathscr L}^{a,b}_{i,j}$ (\ref{lak}) 
are the same except 
$\Rm^{1,1}_{1,1} = {\bf a}^- {\bf a}^+ - {\bf k}^2$ and 
${\mathscr L}^{1,1}_{1,1} = 1$.
This implies that $S(z)^{{\bf a},{\bf b}}_{{\bf i},{\bf j}}$ 
with $(a_\alpha, b_\alpha, i_\alpha, j_\alpha)=(1,1,1,1)$ 
depends on $\epsilon_\alpha=0,1$.
The following table shows such examples, in which 
the case $(\epsilon_1, \epsilon_2, \epsilon_3, \epsilon_4)
=(0,0,0,0)$ is omitted since the expression is too bulky.

\begin{table}[h]
\begin{tabular}{c|c|c|c}
$(\epsilon_1, \epsilon_2, \epsilon_3, \epsilon_4)$ & (0,1,0,1) & (0,1,0,0) & (0,0,0,1)\\
\hline
$S^{1111,0111}_{0121,1101}(z)$ \!\!\!\phantom{$\overset{s}{1}$}& 
$\frac{(1-q^4)z}{(1-qz)(1-q^3z)}$ & 
$\frac{(1-q^4)z(1-q^2-q^4+q^3z)}{(1-qz)(1-q^3z)(1-q^5z)}$ &
$-\frac{q(1-q^4)z(q-z-q^2z+q^4z)}{(1-qz)(1-q^3z)(1-q^5z)}$
\end{tabular}
\end{table}
\end{example}

\section{Generalized quantum group symmetry}\label{sec:qg}

The $S(z)$ constructed in the previous section possesses the  
generalized quantum group symmetry.
Recall that 
$(\epsilon_1,\ldots, \epsilon_n) \in \{0,1\}^n$ is an arbitrary sequence.
Set
\begin{align}\label{qidef}
q_i =(-1)^{\epsilon_i}q^{1-2\epsilon_i},\quad
D_{i,j}&=
\prod_{k\in \{i,i+1\}\cap \{j,j+1\}}(q_k)^{2\delta_{i,j}-1}
\quad (i,j \in \Z_n).
\end{align}

We introduce the $\C(q)$-algebra 
${\mathcal U}_A={\mathcal U}_A(\epsilon_1,\ldots, \epsilon_n)$ 
generated by $e_i, f_i, k^{\pm 1}_i\,(i \in \Z_n)$ obeying the relations
\begin{equation}\label{urel}
\begin{split}
&k_i k^{-1}_i = k^{-1}_ik_i = 1,\quad [k_i, k_j]=0,\\
&k_i e_j = D_{i,j}e_j k_i,\quad k_i f_j = D_{i,j}^{-1}f_j k_i,\quad
[e_i,f_j] = \delta_{i,j}\frac{k_i-k^{-1}_i}{q-q^{-1}}.
\end{split}
\end{equation}
We endow it with the Hopf algebra structure  
with coproduct $\Delta$, counit $\varepsilon$ and 
antipode ${\mathcal S}$ as follows:
\begin{align}
&\Delta k^{\pm 1}_i = k^{\pm 1}_i\otimes k^{\pm 1}_i,\quad
\Delta e_i = 1\otimes e_i + e_i \otimes k_i,\quad
\Delta f_i = f_i\otimes 1 + k^{-1}_i\otimes f_i, \label{Delta}\\
&\varepsilon(k_i) = 1, \quad \varepsilon(e_i) = \varepsilon(f_i) = 0,\quad
{\mathcal S}(k^{\pm 1}_i) = k_i^{\mp 1},\quad
{\mathcal S}(e_i)=-e_ik^{-1}_i, \quad {\mathcal S}(f_i) =  -k_if_i. \nonumber
\end{align}
With a supplement of appropriate Serre relations, 
the homogeneous cases $\epsilon_1=\cdots = \epsilon_n$
are identified with the quantum affine algebras 
\cite{D1,Ji} as
\begin{equation}\label{equiv}
{\mathcal U}_A(0,\ldots, 0) = U_q(A^{(1)}_{n-1}),\quad
{\mathcal U}_A(1,\ldots, 1) = U_{-q^{-1}}(A^{(1)}_{n-1}).
\end{equation}
In general ${\mathcal U}_A(\epsilon_1,\ldots, \epsilon_n)$ 
is an example of generalized quantum groups \cite{H,HY}
including an affinization of quantum super algebra
$sl_q(\kappa,n-\kappa)$.
See \cite[Sec.3.3]{KOS} for more detail.

For the space $\W_l$ (\ref{wl}) and 
a parameter $x$, the following map
$\pi^{(l)}_x: {\mathcal U}_A(\epsilon_1,\ldots, \epsilon_n)
 \rightarrow \mathrm{End}(\W_{l})$ 
gives an irreducible finite dimensional representation\footnote{
Image $\pi^{(l)}_x(g)$ is denoted by $g$ for simplicity.} 
\begin{equation}\label{actsA}
\begin{split}
e_i|{\bf m}\rangle&
= x^{\delta_{i,0}}[m_i]|{\bf m}-{\bf e}_i+{\bf e}_{i+1}\rangle,\\
f_i|{\bf m}\rangle&
= x^{-\delta_{i,0}}[m_{i+1}]|{\bf m}+{\bf e}_i-{\bf e}_{i+1}\rangle,\\
k_i|{\bf m}\rangle&
= (q_i)^{-m_i}(q_{i+1})^{m_{i+1}}|{\bf m}\rangle,
\end{split}
\end{equation}
where 
$[m]= \frac{q^m-q^{-m}}{q-q^{-1}}$ and 
${\bf e}_i = (0,\ldots, 0,\overset{i}{1},0,\ldots, 0) \in \Z^n$.
The vectors $|{\bf m}'\rangle=|m'_1,\ldots, m'_n\rangle$ 
on the rhs of (\ref{actsA}) are to be understood as zero
unless $0 \le m'_i \le 1/\epsilon_i$ for all $1 \le i \le n$.
In the homogeneous case, the representation 
$\pi^{(l)}_x$ is equivalent to
\begin{align*}
\text{degree $l$ symmetric tensor rep. of }   
U_q(A^{(1)}_{n-1})\;\;
\text{for}\;\;\epsilon_1=\cdots = \epsilon_n=0,\\
\text{degree $l$ anti-symmetric tensor rep. of }   
U_{-q^{-1}}(A^{(1)}_{n-1}) \;\;
\text{for}\;\;\epsilon_1=\cdots = \epsilon_n=1.
\end{align*}

Let $\Delta'$ denote the opposite 
(i.e., the left and the right components interchanged) 
coproduct of $\Delta$ in (\ref{Delta}).

\begin{theorem}\label{th:1} 
$\mathrm{(}$\cite[Th.5.1]{KOS}$\mathrm{)}$
For any $l,m \in \Z_{\ge 0}$,
the following commutativity holds:
\begin{align*}
\Delta'(g)S_{l,m}(z)= S_{l,m}(z)\Delta(g)\quad 
\; \forall g \in {\mathcal U}_A(\epsilon_1,\ldots, \epsilon_n),
\end{align*} 
where $\Delta(g)$ and $\Delta'(g)$ stand for the tensor product
representations 
$(\pi^{(l)}_x \otimes \pi^{(m)}_y)\Delta(g)$ 
and $(\pi^{(l)}_x \otimes \pi^{(m)}_y)\Delta'(g)$ of 
(\ref{actsA}) with $z=x/y$.
\end{theorem}

If $\W_l \otimes \W_m$ is irreducible, Theorem \ref{th:1}
characterizes $S_{l,m}(z)$ up to an overall scalar.
Therefore $S_{l,m}(z)$ is identified with the {\em quantum $R$ matrix}
in the sense of \cite{Ji} associated with 
${\mathcal U}_A(\epsilon_1,\ldots, \epsilon_n)$-module $\W_l \otimes \W_m$.
Although we expect that $\W_l \otimes \W_m$ is irreducible 
for arbitrary $(\epsilon_1,\ldots, \epsilon_n)$, 
it has hitherto been proved rigorously only for   
$(\epsilon_1,\ldots, \epsilon_n)$ of the form $(1^\kappa,0^{n-\kappa})$
with $0 \le \kappa \le n$ \cite{KOS}.
Anyway the family $S_{l,m}(z)$ (\ref{sss}) interpolates
the quantum $R$ matrices for 
the symmetric tensor representations of $U_q(A^{(1)}_{n-1})$ and the
anti-symmetric tensor representations of $U_{-q^{-1}}(A^{(1)}_{n-1})$
as the two extreme cases $\kappa=0$ and $n$. 
In \cite[Prop.2.1]{KOS}, it was also shown that 
$S_{l,m}(z)$'s associated with $(\epsilon_1,\ldots, \epsilon_n)$
and $(\epsilon'_1,\ldots, \epsilon'_n)$ are connected by a 
similarity transformation if 
the two sequences are permutations of each other.
Thus one can claim that {\em all} 
the $S_{l,m}(z)$ (\ref{sss}) are equivalent to the 
quantum $R$ matrices of some generalized quantum group.

\section{Combinatorial $R$}\label{sec:cr}
In this section we study $S_{l,m}(z)$ (\ref{wl}) at $q=0$.
Let $(\epsilon_1,\ldots, \epsilon_n) \in \{0,1\}^n$ be an arbitrary sequence
and introduce the {\em crystal}
\begin{align}
B_l = \{{\bf a}=(a_1,\ldots a_n) \in (\Z_{\ge 0})^n\mid 
|{\bf a}| = l,\;0 \le a_i \le 1/\epsilon_i\,(1 \le i \le n)\},
\end{align}
which is a finite labeling set of the basis of $\W_l$ (\ref{wl}).
We identify ${\bf a}=(a_1,\ldots, a_n) \in B_l$
with the depth $n$ column shape tableau
containing $a_i$ dots in the $i$ th box from the top $(1 \le i \le n)$.
See the diagrams given below.
Call the dots in the $i$ th box {\em bosonic} if 
$\epsilon_i=0$ and {\em fermionic} if $\epsilon_i=1$.
Thus there are $l$ dots in the tableau in total among which 
$\epsilon_1a_1+\cdots + \epsilon_na_n$ are fermionic and the rest are bosonic. 

We are going to 
define a map $R=R_{l,m}: B_l \otimes B_m \rightarrow B_m \otimes B_l$
and a function $H=H_{l,m}: B_l \otimes B_m \rightarrow \Z_{\ge 0}$
by combinatorial algorithm, where
$\otimes$ may just be understood as a product of sets.
Thus for a given pair of tableaux ${\bf i}\otimes {\bf j} \in B_l \otimes B_m$,
we are to specify the right hand sides of 
\begin{align}\label{RH}
R({\bf i} \otimes {\bf j}) = {\bf b} \otimes {\bf a} \in B_m \otimes B_l,
\quad
H({\bf i} \otimes {\bf j}) = w \in \Z_{\ge 0}.
\end{align}
For $l \ge m$, it is done by the algorithm (i)--(iii) given below:

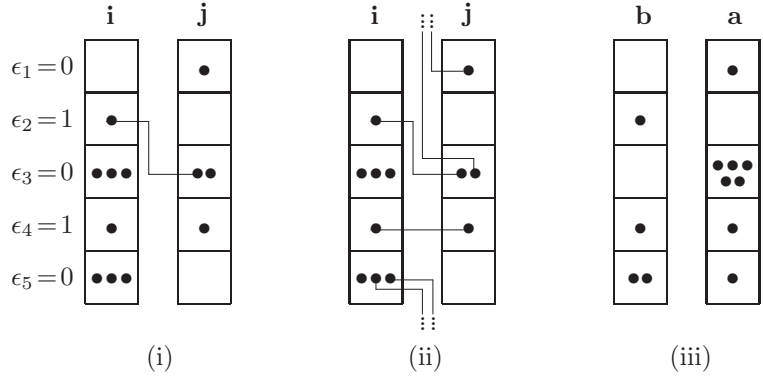
\begin{picture}(300,144)(-55,-24)

\put(-28,87){$\epsilon_1\!=\!0$}
\put(-28,67){$\epsilon_2\!=\!1$}
\put(-28,47){$\epsilon_3\!=\!0$}
\put(-28,27){$\epsilon_4\!=\!1$}
\put(-28,7){$\epsilon_5\!=\!0$}


\put(8,106){${\bf i}$}\put(43,107){${\bf j}$}
\multiput(0,0)(0,20){6}{\put(0,0){\line(1,0){20}}}
\put(0,0){\line(0,1){100}}\put(20,0){\line(0,1){100}}
\put(7.5,67){$\bullet$}
\put(0,40){\put(2,7){$\bullet$}\put(7.5,7){$\bullet$}\put(13,7){$\bullet$}}
\put(7.5,26){$\bullet$}
\put(2,7){$\bullet$}\put(7.5,7){$\bullet$}\put(13,7){$\bullet$}

\put(35,0){
\multiput(0,0)(0,20){6}{\put(0,0){\line(1,0){20}}}
\put(0,0){\line(0,1){100}}\put(20,0){\line(0,1){100}}
\put(7.5,86){$\bullet$}
\put(0,40){\put(5,7){$\bullet$}\put(10,7){$\bullet$}}
\put(7.5,26){$\bullet$}}

\drawline(41,49)(24,49)\drawline(24,49)(24,69)\drawline(24,69)(8,69)

\put(100,0){
\put(8,106){${\bf i}$}\put(43,107){${\bf j}$}
\multiput(0,0)(0,20){6}{\put(0,0){\line(1,0){20}}}
\put(0,0){\line(0,1){100}}\put(20,0){\line(0,1){100}}
\put(7.5,67){$\bullet$}
\put(0,40){\put(2,7){$\bullet$}\put(7.5,7){$\bullet$}\put(13,7){$\bullet$}}
\put(7.5,26){$\bullet$}
\put(2,7){$\bullet$}\put(7.5,7){$\bullet$}\put(13,7){$\bullet$}

\put(35,0){
\multiput(0,0)(0,20){6}{\put(0,0){\line(1,0){20}}}
\put(0,0){\line(0,1){100}}\put(20,0){\line(0,1){100}}
\put(7.5,86){$\bullet$}
\put(0,40){\put(5,7){$\bullet$}\put(10,7){$\bullet$}}
\put(7.5,26){$\bullet$}}

\drawline(43,88)(31,88)\drawline(31,88)(31,103)
\drawline(41,49)(24,49)\drawline(24,49)(24,69)\drawline(24,69)(8,69)
\drawline(47,49)(47,55)\drawline(47,55)(27.5,55)\drawline(27.5,55)(27.5,103)

\multiput(0,0)(0,2){3}{\put(26.2,104.5){.}\put(29.8,104.5){.}}

\drawline(43,28)(8,28)

\drawline(31.5,-4)(31.5,9)\drawline(31.5,9)(17,9)
\drawline(27.5,-4)(27.5,5.5)\drawline(27.5,5.5)(10,5.5)\drawline(10,5.5)(10,9)
\multiput(0,0)(0,-2){3}{\put(26.2,-6){.}\put(30.3,-6){.}}

}

\put(200,0){
\put(8,106){${\bf b}$}\put(43,106){${\bf a}$}
\multiput(0,0)(0,20){6}{\put(0,0){\line(1,0){20}}}
\put(0,0){\line(0,1){100}}\put(20,0){\line(0,1){100}}
\put(7.5,67){$\bullet$}
\put(7.5,26){$\bullet$}
\put(5,7){$\bullet$}\put(10,7){$\bullet$}

\put(35,0){
\multiput(0,0)(0,20){6}{\put(0,0){\line(1,0){20}}}
\put(0,0){\line(0,1){100}}\put(20,0){\line(0,1){100}}
\put(7.5,86){$\bullet$}
\put(0,43){\put(2,7){$\bullet$}\put(7.5,7){$\bullet$}\put(13,7){$\bullet$}}
\put(0,37){\put(5,7){$\bullet$}\put(10,7){$\bullet$}}
\put(7.5,26){$\bullet$}
\put(7.5,7){$\bullet$}}
}

\put(0,-6){\put(23,-17){(i)}\put(122,-17){(ii)}\put(221,-17){(iii)}}
\end{picture}

\begin{enumerate}
\item Choose a dot, say $d$, in ${\bf j}$ and connect it to a 
dot $d'$ in ${\bf i}$ to form a pair.
If $d$ is bosonic (resp. fermionic), 
$d'$ should be the lowest one 
among those located strictly higher (resp. not strictly lower) 
than $d$. 
If there is no such dot, take $d'$ to be the lowest one in ${\bf i}$.
Such a pair is called {\em winding}.
The lines pairing the dots are called {\em $H$-lines}.

\item Repeat (i) for yet unpaired dots until all dots in ${\bf j}$ 
are paired to some dots in ${\bf i}$.

\item Move the $l-m$ unpaired dots in ${\bf i}$ horizontally to ${\bf j}$.
The resulting tableaux define ${\bf b} \otimes {\bf a}$.
$w$ is the winding number (number of winding pairs).
\end{enumerate}

\vspace{0.2cm}
The above example is for $n=5$, $(\epsilon_1,\ldots, \epsilon_5) = (0,1,0,1,0)$,
$B_l \otimes B_m = B_8 \otimes B_4$ and shows 
\begin{align*}
R(01313 \otimes 10210) = 01012 \otimes 10511,\quad
H(01313 \otimes 10210) = 2.
\end{align*}

\renewcommand{\labelenumi}{(\arabic{enumi})}

\begin{remark}\label{re:air} \par \noindent
\begin{enumerate}
\item
In (i) and (ii), the $H$-lines depend on the order of choosing 
the dots from ${\bf j}$.
However, the final result of ${\bf b}\otimes {\bf a}$ 
and $w$ can be shown to be independent of it.

\item
The $H$-lines in the winding case are naturally interpreted as going up 
{\em periodically} along the tableaux.

\item
The condition of being bosonic or fermionic in (i) only refers to $d$ and 
does not concern $d'$. 

\item
When $l=m$, $R$ is trivial in that 
$R({\bf i} \otimes {\bf j}) = {\bf i} \otimes {\bf j}$, but 
$H({\bf i} \otimes {\bf j})$ remains nontrivial.

\item
The algorithm also specifies the number $c_t$ of the $H$-lines 
passing through the border between the $t$ th and 
the $(t\!+\!1)$ th components
in the tableaux ${\bf i}$ and ${\bf j}$ for $t \in \Z_n$.
The winding number is $c_0=c_n$.
For instance in the above diagram (ii), we see
$(c_1,\ldots, c_5)=(1,2,0,0,2)$.
They satisfy the piecewise linear relations:
\begin{alignat}{2}
&\qquad\;\text{$\epsilon_t=0$ case} & 
&\qquad\;\text{$\epsilon_t=1$ case} \nonumber\\
&\begin{cases}
a_t = j_t+(i_t-c_t)_+,\\
b_t = \min(i_t, c_t),\\
c_{t-1}= j_t+(c_t-i_t)_+,
\end{cases}&
&\begin{cases}
a_t = j_t+(i_t-j_t-c_t)_+,\\
b_t = \min(i_t, c_t+j_t),\\
c_{t-1}=(j_t+c_t-i_t)_+,
\end{cases}\label{nzm}
\end{alignat}
where $t \in \Z_n$ and $(x)_+ = \max(x,0)$.
Given ${\bf i}$ and ${\bf j}$, 
one may regard the last rows in (\ref{nzm}) as a closed system of  
piecewise linear equations on $c_1,\ldots, c_n=c_0$ whose 
solution determines ${\bf a}$ and ${\bf b}$ via
the first and the second rows.
We will argue the uniqueness of the solution in the proof of Theorem \ref{th:main}.
\end{enumerate}
\end{remark}

For $l<m$, the algorithm is replaced by the following (i)'--(iii)':

\renewcommand{\labelenumi}{(\roman{enumi})'}

\begin{enumerate}
\item Choose a dot, say $d$, in ${\bf i}$ and connect it to a 
dot $d'$ in ${\bf j}$ to form a pair.
If $d$ is bosonic (resp. fermionic), 
$d'$ should be the highest one 
among those located strictly lower (resp. not strictly higher) than $d$. 
If there is no such dot, take $d'$ to be the highest one in ${\bf j}$.
Such a pair is called winding.

\item Repeat (i)' for yet unpaired dots until all dots in ${\bf i}$ 
are paired to some dots in ${\bf j}$.

\item Move the $m-l$ unpaired dots in ${\bf j}$ horizontally to ${\bf i}$.
The resulting tableaux define ${\bf b} \otimes {\bf a}$.
$w$ is the winding number.
\end{enumerate}

\renewcommand{\labelenumi}{(\roman{enumi})}

Analogue of Remark \ref{re:air} apply to (i)'--(iii)' as well.
It can be shown that $R_{l,m}R_{m,l}=\mathrm{id}_{B_m \otimes B_l}$.
Thus $R$ is a {\em bijection}.
By the definition 
$H_{l,m}({\bf i}\otimes {\bf j})=H_{m,l}({\bf b}\otimes {\bf a})$ 
holds when $R_{l,m}({\bf i}\otimes {\bf j})={\bf b}\otimes {\bf a}$ or equivalently 
$R_{m,l}({\bf b}\otimes {\bf a}) = {\bf i}\otimes {\bf j}$.

The bijective map $R$ and the function $H$ are called 
(classical part of) 
{\em combinatorial $R$} and {\em energy}, respectively.
For $(\epsilon_1,\ldots, \epsilon_n)$ of the form 
$(0^r1^{n-r})$, it was first introduced for $r=0$ and $r=n$ as
Rule 3.10 and Rule 3.11 in \cite{NY} 
in the framework of crystal base theory \cite{Ka} 
of $U_q(\widehat{sl}_n)$, and later for general $r$ in \cite{HI}
based on a realization of $U_q(\mathfrak{gl}(r,n-r))$ crystals in \cite{BKK}.
Note that our algorithm for $r=0$ case, i.e. $\forall \epsilon_i=1$
coincides with \cite[Rule 3.10]{NY}  after reversing the conditions
`higher' and `lower'.
We suppose this is due to the right relation in (\ref{equiv}) indicating the 
interchange of $q=0$ and $q=\infty$ in the two papers.

We define the matrix element of the combinatorial $R$ as
\begin{align}\label{chie}
R^{{\bf a}, {\bf b}}_{{\bf i}, {\bf j}} = \begin{cases}
1 & \mathrm{if}\; R({\bf i} \otimes {\bf j}) = {\bf b} \otimes {\bf a},\\
0 & \mathrm{otherwise}.
\end{cases}
\end{align}

Now we state the main result.

\begin{theorem}\label{th:main}
Let $S^{{\bf a}, {\bf b}}_{{\bf i}, {\bf j}}(z)$ be the element 
(\ref{sss})--(\ref{sabij0}) of 
$S(z)=S_{l,m}(z)$ (\ref{wl}).
Set $R=R_{l,m}$ and $H=H_{l,m}$.
Then the following equality is valid:
\begin{align}\label{LRi}
(1-z)^{\delta_{l,m}}
\lim_{q\rightarrow 0}q^{-(m-l)_+}
S^{{\bf a}, {\bf b}}_{{\bf i}, {\bf j}}(z)= z^{H({\bf i}\otimes {\bf j})}
R^{{\bf a}, {\bf b}}_{{\bf i}, {\bf j}}.
\end{align}
\end{theorem}

\begin{proof}
Setting ${\check S}_{l,m}(z) = PS_{l,m}(z)$ with $P(u \otimes v)=v \otimes u$,
one can show the inversion relation
${\check S}_{l,m}(z) {\check S}_{m,l}(z^{-1}) 
=\rho(z)\,\mathrm{id}_{\W_m \otimes \W_l}$
with an explicit scalar function $\rho(z)$ by using 
(2.30), (2.31),  Proposition 2.1, (3.20), (3.21), 
Theorem 4.1, (6.10), (6.13) and (6.16) in \cite{KOS}.
This reduces the proof to the case $l\ge m$ 
on which we shall concentrate in the sequel.
From \cite[(2.32)]{KO1}, (\ref{Lex}) and (\ref{misto}) we have 
(see also Example \ref{ex:skrb})
\begin{equation}\label{mrei}
\lim_{q\rightarrow 0}
{\mathscr S}^{(\epsilon)\, a, b, c}_{\phantom{(\epsilon)}\,i, j, k}
=\begin{cases}
\lim_{q\rightarrow 0}\Rm^{a,b,c}_{i,j,k}
= \delta^a_{j+(i-k)_+}\delta^b_{\min(i,k)}\delta^c_{j+(k-i)_+}
&(\epsilon=0),\\
\lim_{q\rightarrow 0}{\mathscr L}^{a,b,c}_{i,j,k}
= \delta^a_{j+(i-j-k)_+}\delta^b_{\min(i,k+j)}\delta^c_{(j+k-i)_+}
&(\epsilon=1),
\end{cases}
\end{equation}
which also satisfies (\ref{Lins}).
This is non-vanishing exactly when
(\ref{nzm}) is satisfied after the replacement 
$(\epsilon, a,b,c,i,j,k) \rightarrow (\epsilon_t,a_t,b_t,c_{t-1},i_t,j_t,c_t)$.
Therefore substitution of (\ref{mrei}) into (\ref{sabij0}) leads to 
\begin{align}\label{mkr}
\lim_{q\rightarrow 0}S^{{\bf a}', {\bf b}'}_{{\bf i}, \;{\bf j}}(z)
= \sum_{c_n \ge 0} 
\delta^{{\bf a}'}_{{\bf a}(c_n)}\delta^{{\bf b}'}_{{\bf b}(c_n)}
\delta^{c_n}_{c_0(c_n)} z^{c_n},
\end{align}
where ${\bf a}(c_n) = ({a}_1,\ldots, {a}_n)$,
${\bf b}(c_n) = ({b}_1,\ldots, {b}_n)$ and 
$c_0(c_n)=c_0$ are determined 
from $c_n$ and ${\bf i}, {\bf j}$ uniquely by 
$(\ref{nzm})$ {\em without} the constraint $c_0=c_n$.
The origin of the factor $\delta^{c_n}_{c_0(c_n)}$ is 
the `periodic boundary condition' implied 
by the trace in (\ref{sss})--(\ref{sabij0}).
Thus the proof is reduced to the existence and the uniqueness problem  
of the solution to the equation $c_0(c_n)=c_0$ on $c_n$.

First we assume $l > m$. Then there uniquely exists 
the integer $w\ge 0$ such that $w=c_0(w)$.
In fact such $w$ is given by $w=c_0(0)$.
To see this note that $c_0(s+1)=c_0(s)$ or $c_0(s)+1$ for any $s$ because of 
$(x+1)_+ = (x)_+$ or $(x)_++1$. 
Let $r$ be the smallest non-negative integer such that 
$c_0(r)= w$ and $c_0(r+1)=w+1$.
From (\ref{nzm}) this can happen, either for 
$\epsilon_t=0$ or $1$, only if  
$c_{t-1}=c_t+j_t-i_t$ for all $1 \le t \le n$.
Then $w=c_0(r)$ implies $w=r+|{\bf j}|-|{\bf i}|=r+m-l<r$.
Thus the unique existence of the solution to $w=c_0(w)$ is obvious 
from the following graph.

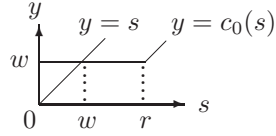
\begin{picture}(70,55)(-140,-11)
\put(0,0){\vector(0,1){30}}
\put(-4,34){$y$}
\put(50,28){$y=c_0(s)$}
\put(-11,14){$w$}
\put(0,16){\line(1,0){40}}
\multiput(38,1.5)(0,3){5}{.}
\drawline(40,16)(48,24)

\multiput(16,1.5)(0,3){5}{.}
\put(14.5,-9){$w$}

\put(15,28){$y=s$}
\put(38,-9){$r$}
\put(0,0){\line(1,1){25}}
\put(0,0){\vector(1,0){55}}\put(60,-3){$s$}
\put(-6,-9){$0$}
\end{picture}

\noindent
Now (\ref{mkr}) reduces to the single term
$\lim_{q\rightarrow 0}S^{{\bf a}', {\bf b}'}_{{\bf i}, \;{\bf j}}(z)
= \delta^{{\bf a}'}_{{\bf a}(w)}\delta^{{\bf b}'}_{{\bf b}(w)}z^{w}$,
where $w$ is the unique solution of $w= c_0(w)$.
It is equal to the energy $H({\bf i} \otimes {\bf j})$
due to Remark \ref{re:air} (5), which also tells that
$R({\bf i} \otimes {\bf j}) = {\bf b}(w) \otimes {\bf a}(w)$. 
Therefore (\ref{LRi}) holds.

Next we consider the case $l=m$.
Again we have (\ref{mkr}) with (\ref{nzm}).
The sum of the first row of (\ref{nzm}) over $1 \le t \le n$
leads to $|{\bf a}|  = |{\bf j}| + f$ with $f\ge 0$.
Due to $|{\bf a}|  =l=m= |{\bf j}| $, $f=0$ must hold 
implying that 
${\bf b}(c_n) \otimes {\bf a}(c_n) = {\bf i} \otimes {\bf j}$.
As for the solution to $c_n=c_0(c_n)$, the previous argument 
tells that it holds for {\em all} $c_n\ge w=c_0(0)= H({\bf i} \otimes {\bf j})$.
Thus the right hand side of (\ref{mkr}) becomes 
$\delta^{{\bf a}'}_{{\bf j}}\delta^{{\bf b}'}_{{\bf i}}
\sum_{c_n \ge H({\bf i} \otimes {\bf j})}z^{c_n}$
in agreement with (\ref{LRi}).
\end{proof}

\begin{example}\label{ex:koyki}
Taking the limit $q\rightarrow 0$ in Example \ref{ex:yum} one has
\begin{align*}
\lim_{q\rightarrow 0} S(z)(|031\rangle \otimes |110\rangle) = 
z^2 |130\rangle \otimes |011\rangle
\end{align*}
for $S(z)= S_{4,2}(z)$. This agrees with the combinatorial $R$\footnote{
The reason for $011 \otimes 130$ rather than 
$130 \otimes 011$ is due to the opposite arrangement of 
${\bf a}$ and ${\bf b}$ in the definitions (\ref{sact}) and  (\ref{chie}).}
and the energy
\begin{align*}
R(031\otimes 110) = 011 \otimes 130,\quad H(031\otimes 110) = 2.
\end{align*}
\end{example}

\begin{example}\label{ex:hrka}
Another check of (\ref{LRi}), where the last line is due to Example \ref{ex:mho}.
\begin{table}[h]
\begin{tabular}{c|c|c|c}
$(\epsilon_1, \epsilon_2, \epsilon_3, \epsilon_4)$ & (0,1,0,1) & (0,1,0,0) & (0,0,0,1)\\
\hline
$R(0121\otimes 1101)$ \phantom{$\underset{s}{\overset{s}{1}}$}
& $0111 \otimes 1111$ & 
$0111 \otimes 1111$ & $0021 \otimes 1201$\\
\hline
$H(0121\otimes 1101)$ \phantom{$\overset{s}{1}$}
& 1 & 1 & 2\\
\hline
$\lim_{q\rightarrow 0}S^{1111,0111}_{0121,1101}(z)$ \!\!\!
\phantom{$\underset{s}{\overset{s}{1}}$}& 
$z$ & 
$z$ &
$0$
\end{tabular}
\end{table}
\end{example}

Let us describe the Yang-Baxter equation satisfied by the combinatorial $R$
as a corollary of (\ref{sybe}).
In order to properly treat the spectral parameter 
we introduce the {\em affine} crystal 
\begin{align*}
\mathrm{Aff}(B_l) =\{{\bf a}[d]\mid {\bf a} \in B_l, d \in \Z\}.
\end{align*}
It allows us to unify the classical part of the combinatorial $R$ 
and the energy $H$ in (\ref{RH}) in the (full) combinatorial $R$
${\mathcal R}={\mathcal R}_{l,m}: 
\mathrm{Aff}(B_l) \otimes \mathrm{Aff}(B_m)
\rightarrow 
\mathrm{Aff}(B_m) \otimes \mathrm{Aff}(B_l)$ as
\begin{align*}
{\mathcal R}({\bf i}[d] \otimes {\bf j}[e]) = 
{\bf b}[e-H({\bf i}\otimes {\bf j})] \otimes {\bf a}[d+H({\bf i}\otimes {\bf j})],
\end{align*}
where ${\bf b} \otimes {\bf a}$ is specified by 
${\bf b} \otimes {\bf a} = R({\bf i}\otimes {\bf j})$.

\begin{corollary}\label{co:sae}
The combinatorial $R$ satisfies the Yang-Baxter equation
\begin{align*}
({\mathcal R}_{l,m} \otimes 1)(1 \otimes {\mathcal R}_{k,m})({\mathcal R}_{k,l}\otimes 1) = 
(1 \otimes {\mathcal R}_{k,l})({\mathcal R}_{k,m}\otimes 1)(1 \otimes {\mathcal R}_{l,m})
\end{align*}
as maps 
$\mathrm{Aff}(B_k)\otimes \mathrm{Aff}(B_l) \otimes\mathrm{Aff}(B_m)
\rightarrow \mathrm{Aff}(B_m)\otimes \mathrm{Aff}(B_l) \otimes\mathrm{Aff}(B_k)$.
\end{corollary}

\begin{example}\label{ex:akna}
Let $n=4$ and $(\epsilon_1,\epsilon_2,\epsilon_3,\epsilon_4) = (1,0,1,0)$.
We apply the two sides of Corollary \ref{co:sae}  on the 
element from
$\mathrm{Aff}(B_4)\otimes \mathrm{Aff}(B_2) \otimes\mathrm{Aff}(B_1)$
in the bottom line. 

\begin{picture}(320,133)(-10,-9)

\put(0,105){0010[$f$]} \put(44,105){0011[$e\!-\!1$]} \put(96,105){1201[$d\!+\!1$]}

\put(20,80){\put(6,3){\vector(2,1){29}}\put(34,3){\vector(-2,1){29}}}
\put(115,83){\vector(0,1){15}}

\put(-6,70){0011[$e\!-\!1$]} \put(50,70){0010[$f$]} \put(96,70){1201[$d\!+\!1$]}

\put(70,45){\put(6,3){\vector(2,1){29}}\put(34,3){\vector(-2,1){29}}}
\put(15,48){\vector(0,1){15}}

\put(-6,35){0011[$e\!-\!1$]} \put(44,35){1210[$d\!+\!1$]} \put(100,35){0001[$f$]}

\put(20,10){\put(6,3){\vector(2,1){29}}\put(34,3){\vector(-2,1){29}}}
\put(114,13){\vector(0,1){15}}

\put(0,0){0211[$d$]} \put(50,0){1010[$e$]} \put(100,0){0001[$f$]}

\put(180,0){
\put(0,105){0010[$f$]} \put(44,105){0011[$e\!-\!1$]} \put(96,105){1201[$d\!+\!1$]}

\put(70,80){\put(6,3){\vector(2,1){29}}\put(34,3){\vector(-2,1){29}}}
\put(15,83){\vector(0,1){15}}

\put(0,70){0010[$f$]} \put(50,70){0211[$d$]} \put(100,70){1001[$e$]}

\put(20,45){\put(6,3){\vector(2,1){29}}\put(34,3){\vector(-2,1){29}}}
\put(113,48){\vector(0,1){15}}

\put(0,35){0211[$d$]} \put(50,35){0010[$f$]} \put(100,35){1001[$e$]}

\put(70,10){\put(6,3){\vector(2,1){29}}\put(34,3){\vector(-2,1){29}}}
\put(15,13){\vector(0,1){15}}

\put(0,0){0211[$d$]} \put(50,0){1010[$e$]} \put(100,0){0001[$f$]}
}
\end{picture}

At the top line the two sides coincide, confirming the Yang-Baxter equation.

\end{example}

\section*{Acknowledgments}
The author thanks Boris G. Konopelchenko, 
Raffaele Vitolo and organizers of {\it Physics and Mathematics of 
Nonlinear Phenomena}, June 20--17 2015 at Gallipoli, Italy, for warm hospitality.
He is also grateful to Masato Okado and Sergey Sergeev for collaboration
in their previous works.
This work is supported by 
Grants-in-Aid for Scientific Research No. 15K13429.

\end{document}